\documentclass[a4paper,11pt]{amsart}
\usepackage{amssymb,amsfonts,amsxtra,
	mathrsfs,graphicx,verbatim,stmaryrd,hyperref,cite,color, tikz-cd, mathabx}
\usepackage[margin=9mm,
marginparwidth=20mm,     
marginparsep=1mm,       
]
{geometry}
\usepackage[all]{xy}
\xyoption{line}
\usepackage{fullpage}
\usepackage{euscript}
\newtheorem{theorem}{Theorem}[section]
\newtheorem{cor}[theorem]{Corollary}
\newtheorem{lem}[theorem]{Lemma}
\newtheorem{prop}[theorem]{Proposition}

\theoremstyle{definition}

\newtheorem{example}[theorem]{Example}
\newtheorem{defi}[theorem]{Definition}
\newtheorem{rem}[theorem]{Remark}

\numberwithin{equation}{section}
\DeclareMathOperator{\colim}{colim}
\DeclareMathOperator{\g}{\mathfrak g}
\DeclareMathOperator{\h}{\mathfrak h}

\DeclareMathOperator{\CE}{CE}

\DeclareMathOperator{\op}{op}
\DeclareMathOperator{\Hom}{Hom}
\DeclareMathOperator{\End}{End}

\DeclareMathOperator{\Mod}{Mod}
\def\MMod{\operatorname{\!-Mod}\nolimits}
\DeclareMathOperator{\Comod}{Comod}
\def\CComod{\operatorname{\!-Comod}\nolimits}
\def\CCtrmod{\operatorname{\!-Ctrmod}\nolimits}
\DeclareMathOperator*{\DGVect}{DGVect}

\DeclareMathOperator*{\DGVectpc}{DGVect_{pc}}
\DeclareMathOperator*{\DGVectop}{DGVect^{op}}
\DeclareMathOperator*{\RHom}{\operatorname{RHom}}

\DeclareMathOperator*{\Uc}{\operatorname{U^c}}
\DeclareMathOperator*{\Li}{\mathcal{L}}
\DeclareMathOperator*{\Tc}{\operatorname{T^c}}
\DeclareMathOperator*{\ad}{\operatorname{ad}}
\DeclareMathOperator*{\Tcon}{\operatorname{T_{con}}}
\DeclareMathOperator*{\Lc}{\mathcal{L}^c}
\DeclareMathOperator*{\Lclf}{\mathcal{L}^c_{lf}}
\DeclareMathOperator*{\Ccon}{C_{con}}
\DeclareMathOperator*{\gcon}{\mathfrak{g}_{con}}
\DeclareMathOperator*{\Lcon}{\mathcal{L}_{con}}
\DeclareMathOperator*{\Ucon}{U_{con}}
\DeclareMathOperator*{\Scon}{S_{con}}
\DeclareMathOperator*{\Sc}{S^c}
\DeclareMathOperator*{\Tcres}{T^c_{res}}

\def\CComodcon{\operatorname{\!-Comod_{con}}\nolimits}
\DeclareMathOperator*{\Bcon}{B_{con}}
\DeclareMathOperator{\PerfII}{Perf^{\mathrm{II}}}
\DeclareMathOperator{\U}{U}
\DeclareMathOperator{\Harr}{Harr}
\DeclareMathOperator{\Tres}{T_{res}\spcheck}
\DeclareMathOperator{\MC}{MC}

\def\ground{\mathbf{k}}
\def\Z{\mathbb{Z}}

\thanks{}

\begin{document}
\begin{abstract}
A Koszul duality-type correspondence between coderived categories of conilpotent differential graded Lie coalgebras and their Chevalley-Eilenberg differential graded algebras is established. This gives an interpretation of Lie coalgebra cohomology as a certain kind of derived functor. A similar correspondence is proved for coderived categories of commutative cofibrant differential graded algebras and their Harrison differential graded Lie coalgebras.
\end{abstract}
\title[Cohomology of Lie coalgebras]{Cohomology of Lie Coalgebras}
\author{Joseph Chuang}
\address{Department of Mathematics, City, University of London, London EC1V 0HB, UK}
	\email{Joseph.Chuang.1@city.ac.uk}
\author{Andrey Lazarev}
	\address{Department of Mathematics and Statistics, Lancaster University, Lancaster LA1 4YF, UK}
	\email{a.lazarev@lancaster.ac.uk}
	
	\author{Yunhe Sheng}
	\address{Department of Mathematics, Jilin University, Changchun 130012, Jilin, China}
	\email{shengyh@jlu.edu.cn}
	
	\author{Rong Tang}
	\address{Department of Mathematics, Jilin University, Changchun 130012, Jilin, China}
	\email{tangrong@jlu.edu.cn}
\thanks{This work was partially supported by EPSRC grants EP/T029455/1, EP/N016505/1, NSFC (12471060,12371029)  and the Fundamental Research Funds for the Central Universities.}
\renewcommand{\thefootnote}{}
\footnotetext{2020 Mathematics Subject Classification.
17B62, 
17B56, 
18N40
}
\keywords{Lie coalgebra, comodule, cohomology, Koszul duality}
\maketitle
\tableofcontents
\section{Introduction}
Cohomology of Lie algebras was introduced over 70 years ago by C. Chevalley and S. Eilenberg and given a standard textbook treatment in the classical treatise  of H. Cartan and S. Eilenberg \cite{CE}. It is conceptually well-understood and is, by now, a standard tool used in Lie theory, representation theory and homological algebra.

By contrast, cohomology of Lie coalgebras has not been as well studied, even though its definition through the standard complex is straightforward, and low-dimensional groups have similar interpretation as in the Lie algebra case \cite{DT}. The reason is that Lie coalgebras and comodules over them are much less understood in general than Lie algebras and their modules, and many of their properties are neither dual nor similar to their Lie counterparts.

To a lesser extent, the same could be said about associative algebras and modules vs coassociative coalgebras and their comodules. However, the structure theory of coassociative coalgebras and comodules has long been well established, and their homological algebra has, somewhat recently, been completely settled in the framework of differential graded (dg) Koszul duality \cite{Posi}. A rather striking consequence of the latter is that from a homotopy-theoretic (or infinity-categorical) perspective, associative algebras and their modules behave similarly to coassociative coalgebras and comodules (rather than dually as might naively be conjectured).

In light of these developments, it is natural to ask whether similar results hold in the Lie and coLie cases. A relevant Koszul duality relates Lie algebras and their modules to cocommutative coalgebras and comodules over them; this theory goes back to Quillen's rational homotopy theory \cite{Qui} and was established in its natural generality by Hinich \cite{Hinich}; it underpins a modern approach to deformation theory. There is also a suitable module-comodule correspondence sketched in \cite[Example 6.6]{Posi}.

Similarly, a Koszul correspondence between Lie coalgebras  and commutative algebras was constructed in \cite{LM} in the context of disconnected rational homotopy theory. The Lie coalgebras considered in op. cit. were \emph{conilpotent}, i.e. filtered colimits of finite dimensional Lie coalgebras whose dual Lie algebras are nilpotent. This excludes large classes of Lie coalgebras (e.g. those whose duals are semisimple Lie algebras). There is a good reason (discussed later in the present paper, cf. Example \ref{ex:sl2}) why Koszul duality cannot be extended to such Lie coalgebras. However, even for conilpotent Lie coalgebras the Koszul correspondence on the module-comodule level is far from obvious, the reason being that, unlike in the associative case, not every comodule over a conilpotent Lie coalgebra is itself conilpotent. This is the main technical difficulty that we grapple with in this paper.

Given a dg Lie coalgebra $\g$, we consider its category of \emph{locally finite} comodules, i.e. comodules that are unions of finite-dimensional ones. This category is isomorphic to the category of comodules over $\U^c(\g)$, the universal enveloping coalgebra of $\g$. As such, it has the structure of a model category whose homotopy category is called the \emph{coderived category} of $\g$. Let $\CE(\g)$ be the Chevalley-Eilenberg complex of $\g$ which is a dg {commutative} algebra. The category of $\CE(\g)$-modules has a model category of its own, and thus, a homotopy category, called the \emph{coderived} category of $\CE(\g)$; note that this model structure is finer than the ordinary projective one whose homotopy category is the \emph{derived} category of $\CE(\g)$.

Our first main result, Theorem \ref{thm:coreflective} is that there is a coreflective Quillen adjunction between comodules over $\g$ and modules over $\CE(\g)$, provided that $\g$ is conilpotent; it becomes a Quillen equivalence under the additional assumption that $\g$ is non-negatively graded (Theorem \ref{thm:nonneg}).
As a consequence, we are able to interpret the Chevalley-Eilenberg cohomology of $\g$ with coefficients in a $\g$-comodule as a derived functor (Corollary \ref{cor:derived}). It is quite surprising that the assumption of non-negative grading needs to be imposed for the Quillen equivalence statement, however the example of the coderived category of the polynomial de Rham algebra, cf. Example \ref{ex:deRham} below, shows that it is unavoidable.

Our second main result gives a similar treatment for modules over \emph{cofibrant} dg commutative algebras and comodules over their Harrison dg Lie coalgebras; here the condition of cofibrancy is parallel to the condition of conilpotency of Lie coalgebras and the result does not hold without it. This is Theorem \ref{thm:harrison}.

Dg Lie coalgebras appear naturally in rational homotopy theory as primitive elements in the cobar constructions of Sullivan minimal models of topological spaces and in connection with the problem of finding homotopy periods, cf \cite{SW, LM}. Modules over Sullivan minimal models correspond, according to our results, to comodules over the corresponding Harrison {dg Lie} coalgebras and, in the case of smooth manifolds, give rise to modules over the corresponding de Rham algebras. Such modules are known under various other names: flat superconnections, $\infty$-local systems or cohomologically locally constant sheaves, and they have been well-studied, e.g. \cite{SW, AVV, BS, CHL} and are, therefore, abundant in nature. It would be interesting to make a systematic study of the relationship between $\infty$-local systems and comodules over  the corresponding Lie coalgebras but we will leave it for future work.

The organization of the paper is as follows. Section \ref{sec:2} introduces notations and conventions used in this paper. Section \ref{section:preliminaries} is a crash course on the fundamentals of Lie coalgebras, including the construction of a universal enveloping coalgebra and the analogue of the PBW theorem. We put the focus on  conceptual explanation and on instructive examples; a much more detailed treatment  is given in \cite{Mic}. Of some independent interest here is the notion of a Lie-conilpotent {coassociative} coalgebra and the construction of a cofree Lie-conilpotent {coassociative} coalgebra sitting between a cofree {coassociative} coalgebra and a  cofree conilpotent {coassociative} coalgebra, cf. Proposition \ref{prop:Lieconilpotent} below.
 Section \ref{section:Koszul1} contains the Koszul duality statement for conilpotent Lie coalgebras and all necessary preliminaries. A much simpler result involving conilpotent comodules is also proved; this is Theorem \ref{thm:conilduality}. The Koszul equivalence for cofibrant commutative algebras is proved in Section \ref{section:cofibrant} as well as a much simpler statement involving conilpotent comodules over the Harrison {dg} Lie coalgebra (Theorem \ref{thm:conil2}). At the end of this section, a relationship between Lie comodules and $\infty$-local systems is outlined.

\section{Notation and conventions}\label{sec:2} We will work in the category
$\DGVect$
of differential graded (dg) vector spaces over a fixed field $\ground$ of characteristic zero.
The objects of
$\DGVect$
are $\ground$-vector spaces with cohomological $\mathbb{Z}$-grading, supplied with a differential of degree $1$; morphisms are linear maps commuting with the differential. The degree of an element $x$ in a graded vector space will be denoted by $|x|$.
Given a dg vector space $V$, its suspension $\Sigma V$ is another dg vector space with $(\Sigma V)^i=V^{i+1}$. Under the functor $*$ of $\ground$-linear duality, we have $(V^*)^i:=(V^{-i})^*$.

We will use the language of pseudocompact (pc) vector spaces. A pc (graded) vector space is a projective limit of finite-dimensional vector spaces and morphisms of pc vector spaces are always assumed to be continuous with respect to the linearly compact topology of the projective limit.   Similarly a pc dg vector space is a pc graded vector space supplied with a graded differential of cohomological degree 1 squaring to zero. The category of pc dg vector spaces will be denoted by $\DGVectpc$. The functor of continuous linear duality establishes an anti-equivalence between the categories of (discrete) graded vector spaces and pc graded vector spaces and so the category $\DGVectop$, the opposite to $\DGVect$, is equivalent to $\DGVectpc$.  Since $\DGVect$ is symmetric monoidal with respect to the graded tensor product, so is $\DGVectop$ and therefore, $\DGVectpc$. The monoidal product of two pc graded spaces $V$ and $U$ will be denoted by $V\otimes U$; note that this is \emph{not} the tensor product of $V$ and $U$ viewed as discrete graded vector spaces unless either $V$ or $U$ is finite-dimensional.

A dg algebra is just a monoid in $\DGVect$; it is thus a dg vector space
$A$ supplied with an associative and unital multiplication $A\otimes A\to A$. Similarly, a dg pc algebra is a monoid in $\DGVectpc$; by the anti-equivalence of the latter category with $\DGVect$, a dg pc algebra is the same as a comonoid in $\DGVect$, i.e. a dg coalgebra. A dg Lie algebra is a Lie object in $\DGVect$, i.e. a dg vector space $\g$ supplied with an anti-commutative bracket $\g\otimes\g\to \g$ satisfying the Jacobi identity; similarly, a dg pc Lie algebra is a Lie object in $\DGVectpc$. A dg Lie coalgebra is a vector space $\g$ with the cobracket $\g\to \g\otimes \g$ making the dual space $\g^*$ into a Lie object in $\DGVectpc$. Since a Lie object in $\DGVectpc$ is equivalent to a Lie co-object in $\DGVect$, then a  dg pc Lie algebra carries the same information as a dg Lie coalgebra.

Given a dg algebra or dg Lie algebra $A$, we consider its category $\Mod(A)$ of left dg $A$-modules; these are dg vector spaces $V$ supplied with an action map $A\otimes V\to V$ satisfying the usual axioms. If $A$ is a dg coalgebra or a dg Lie coalgebra, we can similarly denote by $\Comod(A)$ the category of left dg $A$-comodules. We also need the notion of a bicomodule over a dg coalgebra $A$; it is a dg vector space supplied with a left coaction $V\to A\otimes V$ and a right action $V\to V\otimes A$, compatible in the usual way. For a dg algebra $A$, a twisted $A$-module is a dg module whose underlying graded $A$-module is free, i.e. it is isomorphic to $A\otimes V$ for a graded vector space $V$; if $\dim V<\infty$, we say that the twisted module is finitely generated. Similarly, a twisted comodule over a dg coalgebra $C$ is one whose underlying graded $C$-comodule is cofree, i.e. of the form $C\otimes V$ for a graded vector space $V$. The category $\Mod(A)$ for a dg algebra $A$ has a model structure with weak equivalences being quasi-isomorphisms; similarly for a dg coalgebra $A$, the category $\Comod(A)$ has a model structure with weak equivalences being maps with coacyclic cofiber \cite[Section 8]{Posi}.

To alleviate the language, from now on, we will generally omit the modifier `dg' in the rest of the paper; e.g. `Lie coalgebra' will mean `dg Lie coalgebra', modules and comodules will likewise be assumed to be dg etc.

\section{Lie coalgebras, their comodules and universal enveloping coalgebras}\label{section:preliminaries}

\subsection{Universal enveloping coalgebras of Lie coalgebras}
Let $\g$ be a Lie coalgebra. Then its linear dual $\g^*$ is a pc Lie algebra and the functor of (continuous) linear duality establishes an anti-equivalence between Lie coalgebras and pc Lie algebras.  Furthermore, there is a functor associating to any coassociative coalgebra its cocommutator Lie coalgebra. This functor $C\mapsto L(C)$ is easiest to describe using the anti-equivalence mentioned above. Namely, $L(C)$ is the Lie coalgebra that is continuous linear dual to the pc Lie algebra $L(C^*)$, the commutator Lie algebra of the pc algebra $C^*$. The following result holds:

\begin{theorem}
	The functor $C\mapsto L(C)$ from {coassociative coalgebras} to {Lie coalgebras} admits a right adjoint, $\g\mapsto \Uc(\g)$, called the \emph{universal enveloping coalgebra} of the Lie coalgebra $\g$.
\end{theorem}
\begin{proof}
	In the non-dg case this is established in \cite{Mic} and the extension to the dg case is straightforward.
\end{proof}	
\begin{rem}
	The construction of $\Uc(\g)$ is given explicitly as follows. First assume that $\g$ is locally finite, in other words, it is the union of its finite dimensional Lie subcoalgebras; if this is not the case, then replace $\g$ by its largest locally finite Lie subcoalgebra.
	It is thus sufficient to define $\Uc(\g)$ for a finite dimensional
	$\g$ since in general
	$\Uc(\g)\cong \colim_{\alpha} \Uc({\g}_{\alpha})$
	where the colimit is taken over all finite dimensional Lie subcoalgebras of $\g$.
	For a finite dimensional $\g$, the coalgebra $\Uc(\g)$ is defined as $\U^0(\g^*)$, the Sweedler dual of the ordinary enveloping algebra $\U(\g^*)$ of the Lie algebra $\g^*$; see \cite[Chapter 6]{sweedler} regarding the notion of the Sweedler dual. Recall that for an algebra $A$, its Sweedler dual is the set of linear functions $f$ on $A$ such that $\ker f$ contains an ideal of finite codimension; in the dg setting this is modified by requiring that $\ker f$ contains a \emph{dg} ideal of finite codimension. More precisely, the Sweedler dual of an associative algebra $A$ is given by ${A}^0:=\varinjlim_\alpha (A/I_{\alpha})^*$ where the direct limit is taken over all two-sided dg ideals of finite codimension.
	Equivalently, $\Uc(\g)$ is defined by setting
	$[\Uc(\g)]^*:=\check{{\U}}(\g^*)$,
	the pseudo-compact completion of $\U(\g^*)$. The latter is defined for any  algebra $A$ as $\check{A}:=\varprojlim_\alpha A_\alpha$ where the inverse limit is taken over all finite-dimensional quotient algebras $A_\alpha$ of $A$.
\end{rem}
\begin{rem}
	The coalgebra $\Uc(\g)$ is a Hopf algebra. Indeed, let $\g$ be finite-dimensional, then the ordinary enveloping algebra
	$\U(\g^*)$ of $\g^*$ is known to be a Hopf algebra with the coproduct $\Delta:\U(\g^*)\to\U(\g^*)\otimes\U(\g^*)$ and
	this map induces a coproduct on the pseudocompact completion
	$\check{\U}(\g^*) \to\check{\U}(\g^*)\hat{\otimes}\check{\U}(\g^*)$
		making $\check{\U}(\g^*)$ into a pseudocompact bialgebra and dualizing, we see that $\Uc(\g)$ is a bialgebra; the existence of an antipode is likewise straightforward.
\end{rem}
\begin{rem}
	Corresponding to the identity map on $\Uc(\g)$ is a map of Lie coalgebras $\Uc(\g)\to\g$ (or its more familiar dual $\g^*\to \check{\U}(\g^*)$). According to \cite{Mic}, the latter map is surjective if and only if $\g$ is locally finite. In contrast, for nonlocally-finite  Lie coalgebras, this map can even be zero; this happens when a Lie coalgebra has no nontrivial finite-dimensional Lie subalgebras. For example, such is the Lie coalgebra whose linear dual is the Lie algebra of derivations of the formal power series ring $\ground[[z]]$.
\end{rem}
Recall that a comodule over a Lie coalgebra $\g$ is a vector space $M$ together with a coaction $M\to \g\otimes M$  satisfying a certain comodule analogue of the coJacobi identity, cf. for example \cite[Definition 2.3]{DT}. The above coaction is equivalent to a map $\g^*\otimes M^*\to M^*$ and if $\dim M< \infty$   this is further equivalent to a map of pc Lie algebras $\g^*\to \End(M^*)\cong \End(M)$ specifying a continuous action of the pc Lie algebra $\g^*$ on the finite-dimensional vector space $M$.

Recall that modules over Lie algebras are in one-to-one correspondence with modules over their universal enveloping algebras. For Lie coalgebras, the analogous statement is formulated as follows.

\begin{prop}
	Let $\g$ be a Lie coalgebra. Then $\Uc(\g)\CComod$, the category of  comodules over $\Uc(\g)$ is isomorphic to the category of locally finite $\g$-comodules, i.e. those $\g$-comodules which are unions of its finite-dimensional subcomodules.
\end{prop}
	
\begin{proof}
Let $M$ be a finite dimensional $\g$-comodule determined by the coaction  $M\to \g\otimes M$. This coaction is equivalent to a \emph{continuous} map $\g^*\to\End(M)$, and the comodule condition translates into the latter map being a homomorphism  of  pc Lie algebras. This, by the universal property of the universal enveloping algebra, is equivalent to a continuous algebra map $[\Uc(\g)]^*\to \End(M)$ or to a $\Uc(\g)$-comodule structure on $M$. So we obtain  an isomorphism between the categories of finite-dimensional $\g$-comodules and finite-dimensional $\Uc(\g)$-comodules. This extends to an isomorphism between locally finite $\g$-comodules and	locally finite $\Uc(\g)$-comodules (which are all $\Uc(\g)$-comodules).
\end{proof}	
\begin{rem}
	Since $\Uc(\g)$  is a Hopf algebra, it follows that locally finite $\g$-comodules have the structure of the tensor category; given two locally finite $\g$-comodules  $M$ and $N$, their tensor product $M\otimes N$ as well as $\Hom(M,N)$ have structures of (locally finite) $\g$-comodules.
\end{rem}
\begin{example}
	Let $\g$ be the dual of a semisimple Lie algebra; assume that $\ground$ is algebraically closed. Then every finite-dimensional (and therefore locally finite) representation of $\g^*$ is semisimple, so every $\Uc(\g)$-comodule and it follows that the coalgebra $\Uc(\g)$ is cosemisimple. As such, it is a sum of simple (matrix) coalgebras $\sum_{\alpha} \End(V_\alpha)$, where $V_\alpha$ ranges through isoclasses of irreducible representations of $\g^*$.
	
	It is instructive to compare $[\Uc(\g)]^*\cong \check{\U}(\g^*)$ with $\U(\g^*)$,
the usual universal enveloping algebra of $\g^*$. By the argument above, we have
$\check{\U}(\g^*)\cong \prod_{\alpha}  \End(V_\alpha)$. The completion map $\U(\g^*)\mapsto \check{\U}(\g^*)$  is determined by the collection of action maps $\U(\g^*)\to \End(V_\alpha)$ associated with the given irreducible representations of $\g^*$ on $V_\alpha$.  It is well-known that any nonzero element in $\U(\g^*)$ acts nontrivially in some
finite-dimensional representation of $\g^*$, and thus the completion map embeds
$\U(\g^*)$ as a subalgebra of $\check{\U}(\g^*)$ containing $\sum_\alpha\End(V_\alpha)$. In other words, we have the following inclusions of associative algebras:
\[
\sum_\alpha\End(V_\alpha)\subset \U(\g^*)\subset \prod_\alpha\End(V_\alpha)\cong \check{\U}(\g^*).
\]	
\end{example}
\begin{example}\label{ex:abelian}
	Let $\g$ be the one-dimensional abelian Lie coalgebra. In that case $\Uc(\g)$ is the cofree coalgebra on $\ground$. Again, let us assume that $\ground$ is algebraically closed. In that case $\check{\U}(\g^*)$ is the pc completion of the polynomial algebra $\ground[x]$, and it is isomorphic to $\prod_{\alpha}\ground[[x_\alpha]]$ where $\alpha$ runs through all points of $\ground$. The completion map $\ground[x]\to \prod_{\alpha}\ground[[x_\alpha]]$ takes $x\in\ground[x]$ to $\prod_{\alpha}(
	x_\alpha-\alpha)$.
	
	It can also be expressed as the convolution algebra $\Hom(\ground[\ground],\ground[[x]])\cong \operatorname{Map}(\ground,\ground[[x]])$ where $\ground[\ground]$ is the group algebra of the additive group of $\ground$ and $\operatorname{Map}(\ground,\ground[[x]])$ stands for the set of maps of sets $\ground\to\ground[[x]]$. Indeed, the algebra $\operatorname{Map}(\ground,\ground[[x]])$ with the pointwise multiplication is clearly isomorphic to $\prod_{\alpha}\ground[[x_\alpha]]$.
	
	Moreover, $\check{\U}(\g^*)$ is a Hopf algebra where its diagonal is induced from the multiplication on $\ground[\ground]$ and the diagonal on $\ground[x]$ with $x$ being primitive.
	
	If $\ground$ has characteristic zero (as per our standing assumption), we can also express $\Uc(\g)$ as $\ground[\ground]\otimes\ground[x]$. As a Hopf algebra, it is a tensor product of Hopf algebras $\ground[\ground]$ and $\ground[x]$ where $\ground[\ground]$ has the diagonal as a group algebra and the diagonal in $\ground[x]$ is specified by requiring that $x$ be primitive.

\end{example}

\begin{example}\label{ex:cofree} This is a generalization of Example \ref{ex:abelian}.
	Let $V$ be a vector space. There is a functor $V\mapsto \Lc(V)$ associating to $V$ the cofree Lie coalgebra, cf. \cite{Mic}. The functor $V\mapsto \Lc(V)$ is right adjoint to the forgetful functor from Lie coalgebras to vector spaces. The union of finite-dimensional Lie subcoalgebras in $\Lc(V)$ is the cofree \emph{locally finite} Lie coalgebra on $V$;
	 we denote it by $\Lclf(V)$. Then $\Uc(\Lc(V))\cong \Uc(\Lclf(V))\cong \Tc(V)$ is the cofree coassociative coalgebra on $V$, see \cite[Theorem 6.4.1]{sweedler} regarding the construction of $\Tc(V)$. For $V$ finite-dimensional, it can be described as the Sweedler dual to the free algebra ${\rm T}(V^*)$ and in general, as the colimit of the coalgebras $\Tc(W)$ where $W$ ranges through finite-dimensional subspaces of $V$. Equivalently (for a finite-dimensional $V$), $\Tc(V)$ is described by requiring that its pc dual algebra is $\check{\rm T}(V^*)$, the completion of the free algebra ${\rm T}(V^*)$ by all two-sided ideals of finite codimension.
\end{example}	

We need some standard results about the relationship between bicomodules over universal enveloping coalgebras and one-sided modules. Let $C$ be a dg Hopf algebra with a bijective antipode $S$ and consider the map $m: C\otimes C^{\op}\to C$ so that $a\otimes b\mapsto a S^{-1}(b)$. Then $m$ is a coalgebra map and so, every $C$-bicomodule $M$ (which is by definition a $C\otimes C^{\op}$-comodule) can be viewed as a $C$-comodule by corestriction along $m$. We will denote this $C$-comodule by $M^{\ad}$. The following result holds; it will later be applied when $C$ is an enveloping coalgebra of a Lie coalgebra.

\begin{prop}\label{prop:onetwosided}
	Let $C$ be a dg Hopf algebra and $M$ be a $C\otimes C^{\op}$-comodule. Then there is a natural isomorphism
	\[
	\operatorname{RHom}_{C\otimes C^{\op}\CComod}(M,C)\cong\operatorname{RHom}_{C\CComod}(M^{\ad},\ground).
	\]
\end{prop}
\begin{proof}
	Passing to the dual pc algebra $A:=C^*$ and its pc module $N:=M^*$, we can rewrite the desired isomorphism as follows:
	\[
	\operatorname{RHom}_{A\otimes A^{\op}\MMod}(A, N)\cong \operatorname{RHom}_{A\MMod}(k,N^{\ad}).
	\]
	This is well-known in the context of ordinary derived functors for bimodules over Hopf algebras and the proof carries over to the pc context; we will indicate the main steps.
	
	Step 1: there is an isomorphism of $A$-bimodules
	\begin{equation}\label{eq:iso1}
	(A\otimes A^{\op})\otimes_A\ground \to A
	\end{equation}
	where $A$ acts on the right hand side of (\ref{eq:iso1}) by left and right multiplication and it acts on $A\otimes A^{\op}$ by the map $m^*:A\to A\otimes A^{\op}$ dual to $m$. The map  $(A\otimes A^{\op})\otimes_A\ground\to A$ is given by $a\mapsto a\otimes 1\otimes 1$ and the inverse map is given by $a\otimes b\otimes 1\mapsto ab$.  This is verified in \cite[Lemma 9.4.2]{Wi} in the discrete case and carries over verbatim to the pc case.
	
	Step 2: the right $A$-module $A\otimes A^{\op}$ given by $m^*$ is isomorphic to
	 $A_A\otimes A_{\ground}$ where $A_A$ indicates the regular right $A$-module and $A_{\ground}$ -- the vector space $A$ with the trivial (i.e. factoring through the augmentation) $A$-module structure. The isomorphism $A_A\otimes A_{\ground}\to A\otimes A^{\op}$ is given by $a\otimes b\mapsto (1\otimes S^*)(a\otimes b)$. In particular, $A\otimes A^{\op}$ with this $A$-module structure, is (topologically) free (and so cofibrant; this is equivalent to the dual $C$-comodule $C\otimes C^{\op}$ being cofree and therefore fibrant). This  is verified in \cite[Lemma 2.2]{BJ} in the discrete case and carries over verbatim to the pc case.
	
	 	 Step 3. We have the following  isomorphisms:
	 \begin{align*}
	 \operatorname{RHom}_{A\otimes A^{\op}}(A,N)&\cong \operatorname{RHom}_{A\otimes A^{\op}}((A\otimes A^{\op})\otimes_A\ground, N) \quad{\text{ by step 1}}\\
	 &\cong \operatorname{RHom}_{A\otimes A^{\op}}((A\otimes A^{\op})\otimes^{\mathbb{L}}_A\ground,N) \quad{\text{ by step 2}}\\
	 &\cong \operatorname{RHom}_{A}(\ground, N^{\ad}),
	 \end{align*}
which finishes the proof.
\end{proof}
\begin{rem}
The graded vector space $\operatorname{RHom}_{C\otimes C^{\op}\CComod}(M,C)$ is isomorphic to $\operatorname{HH}(C,M)$, the Hochschild cohomology of $C$ with coefficients in $M$, cf. for example \cite{Kel, GHL} regarding the Hochschild cohomology of coalgebras.
\end{rem}
\subsection{Conilpotent Lie coalgebras}
Examples discussed above show that universal enveloping coalgebras of Lie coalgebras exhibit rather different behaviour than ordinary universal enveloping algebras of Lie algebras. There is, however, a class of Lie coalgebras admitting a different definition of a universal enveloping  `coalgebra', which is rather close to the classical picture.
\begin{defi}
	A finite-dimensional  Lie coalgebra $\g$ is called \emph{conilpotent} if its linear dual Lie algebra $\g^*$ is nilpotent (i.e. there exists $N\in \mathbb{N}$ such that any Lie word in $\g^*$ of length $>N$ is zero). A (not necessarily finite-dimensional) Lie coalgebra is conilpotent if it is a union of finite-dimensional conilpotent Lie coalgebras.
	
	A finite-dimensional comodule $M$ over $\g$ is conilpotent if $M^*$ is nilpotent over $\g^*$, i.e. there exists $N\in \mathbb{N}$ such that any word in $\g^*$ of length $>N$ acts as zero in $M^*$. A (not necessarily finite dimensional) $\g$-comodule $M$ is conilpotent if it is a union of finite-dimensional conilpotent $\g$-comodules.
\end{defi}
There is a parallel definition of a conilpotent coassociative coalgebra and a corresponding conilpotent comodule, cf. \cite{Posisurvey} for details. Note that in older sources such as \cite{sweedler} a conilpotent coalgebra is called \emph{irreducible}.

\begin{defi}
Let $C$ be a coaugmented finite-dimensional coalgebra; denote by $\overline{C}$ the cokernel of the coaugmentation $\ground\to C$; it is thus a noncounital coalgebra. Then $C$ is called conilpotent if $\overline{C}^*$ is nilpotent, i.e. there exists $N\in \mathbb{N}$ such that any product of more than $N$ elements in $\overline{C}^*$ vanishes. A not necessarily finite-dimensional coalgebra $C$ is conilpotent if it is a union of finite-dimensional conilpotent coalgebras. 	
	\end{defi}
Given a coalgebra $C$ with a coaugmentation $\epsilon:\ground\to C$, there is a unique maximal conilpotent subcoalgebra $\Ccon\hookrightarrow C$ with a compatible coaugmentation; this is just the union of all conilpotent subcoalgebras of $C$ containing the image of $\epsilon$. From the point of view of dual pc algebras, $(\Ccon)^*$ is the completion of $C^*$ at the ideal of the augmentation (dual to the coaugmentation of $C$). For a Lie coalgebra $\g,$ the union of its conilpotent Lie subcoalgebras is itself a conilpotent Lie subcoalgebra $\gcon$ and its dual $\gcon^*$ is the completion of $\g^*$ with respect to the wordlength of Lie monomials.

The following result is obtained using standard techniques, see e.g. \cite[Corollary 8.0.9]{sweedler} in the associative case.
\begin{prop}
	The functor $C\mapsto \Ccon$ from the category of coaugmented coalgebras to conilpotent coalgebras is right adjoint to the natural inclusion functor. The functor $\g\mapsto \gcon$ from Lie coalgebras to conilpotent Lie coalgebras is likewise right adjoint to the inclusion functor.
\end{prop}

\begin{rem}
	One can similarly define a conilpotent finite-dimensional comodule $M$ over a conilpotent coalgebra $C$ by the condition that any product of sufficiently many elements of $C^*$ acts as zero on $M^*$ and, furthermore, extend this definition to arbitrary comodules. One quickly observes however, that \emph{any} comodule over a conilpotent coalgebra is itself conilpotent making such a definition extraneous. In contrast, not all comodules over a conilpotent Lie coalgebra are conilpotent. Indeed, let $\g$ be the one-dimensional Lie coalgebra concentrated in degree zero and the zero bracket. A $\g$-comodule is just a vector space $V$ together with an endomorphism $f$ of $V$, which can be arbitrary. If $\dim V<\infty$ then the corresponding $\g$-comodule is conilpotent if and only if the endomorphism $f$ is nilpotent. This phenomenon is related to the fact that the universal enveloping coalgebra of a conilpotent Lie coalgebra need not itself to be conilpotent.
\end{rem}
A closely related notion to conilpotency (both in the associative and Lie algebras setting) is that of \emph{pronilpotency}.
\begin{defi}\label{def:pron}
	An augmented algebra $A$ with the augmentation ideal $I$ is pronilpotent  if it is complete with respect to the $I$-adic filtration:
	\[A\cong \varprojlim_n A/I^n.\]
	Similarly, a Lie algebra $\g$ is pronilpotent if it is a projective limit of its nilpotent quotients; more precisely, denoting by $\g_n$ the Lie ideal in $\g$ generated by Lie bracket of length $n$, we have:
	\[
	\g\cong \varprojlim_n \g/\g_n.
	\]
\end{defi}
\begin{prop}
	Let $C$ be a conilpotent coaugmented coalgebra. Then $C^*$ is pronilpotent. Similarly, if $\g$ is a conilpotent Lie coalgebra, then $\g^*$ is pronilpotent.
\end{prop}
\begin{proof}
	Consider the collection of two-sided ideals of finite codimension of $C^*$. They form a category with morphisms being inclusions of ideals. We have $C^*=\varprojlim_\alpha C^*/I_{\alpha}$ since $C^*$ is a pc algebra. Note that by our assumption on $C$, all algebras $C^*/I_{\alpha}$ have nilpotent augmentation ideals. Now consider the category formed by the powers $I^n$ of the augmentation ideal $I$ in $C^*$ and one map between any two objects corresponding to the natural inclusion. This is cofinal inside the category of all ideals of finite codimension and so we have
	\begin{align*}
	C^*&\cong \varprojlim_\alpha C^*/I_{\alpha}    \cong \varprojlim_n\varinjlim_{I_{\alpha_n}\subset I^n} C^*/I_{\alpha_n} \cong \varprojlim_nC^*/I^n.
	\end{align*}
So, $C^*$ is pronilpotent as claimed. The proof for a conilpotent Lie coalgebra is similar.	
\end{proof}	
\begin{rem}
	Of course, not every pronilpotent algebra or Lie algebra is pc. For example, a countably-dimensional abelian Lie algebra is pronilpotent according to Definition \ref{def:pron} but it is not pc. Such pronilpotent (Lie) algebras cannot be duals to (Lie) coalgebras.
\end{rem}
\begin{example}
	The cofree conilpotent coalgebra on a vector space $V$ is the tensor coalgebra $\Tcon(V) = \oplus_{n\geq 0 } V^{\otimes n}$. The diagonal is given by the formula
	\begin{eqnarray*}
 && \Delta(v_1\otimes \ldots \otimes v_n)\\
  &=&1\otimes (v_1\otimes\ldots \otimes v_n)+\sum_{k=1}^{n-1} (v_1\otimes \ldots\otimes v_k)\otimes(v_{k+1}\otimes\ldots\otimes v_n)+(v_1\otimes\ldots \otimes v_n)\otimes 1.
\end{eqnarray*}
	where $v_i\in V, i=1,2,\ldots$. It is easy to see that $\Tcon(V)\cong [\Tc(V)]_{\operatorname{con}}$. In the case $\dim V<\infty$, the dual pc algebra to $\Tcon(V)$ is $\hat{T}(V^*)$, the completion of $T(V^*)$ at the maximal ideal generated by $V^*$. As the name indicates, the functor $V\mapsto \Tcon(V)$ is right adjoint to the forgetful functor from conilpotent coalgebras to vector spaces.
	
	The coalgebra $\Tcon(V)$ is in fact a Hopf algebra. Its dual Hopf algebra (when $V$ is finite-dimensional) is $\hat{T}(V^*)$  with the diagonal  determined by the requirement that the elements in $V^*$ be primitive.
	
	Similarly, the cofree conilpotent Lie coalgebra $\Lcon(V)$ on a vector space $V$ is given by $[\Lc(V)]_{\operatorname{con}}$, the maximal conilpotent Lie coalgebra in the cofree Lie coalgebra on $V$. It is well-known that (under the condition $\dim V<\infty$) the pc Lie algebra $[\Lcon(V)]^*$ can be identified with the Lie subalgebra of primitive elements in the Hopf algebra $\hat{T}(V^*)$ described above.
\end{example}
\begin{defi}
	Let $\g$ be a Lie coalgebra. Then its conilpotent universal enveloping coalgebra $\Ucon(\g)$ is defined as $\Ucon(\g):=[\Uc(\g)]_{\operatorname{con}}$.
\end{defi}
\begin{rem}\label{rem:conil}
	By definition, $\Ucon(\g)$ is conilpotent. If $\g$ is finite-dimensional, then $[\Ucon(\g)]^*\cong\hat{ \U}(\g^*)$, the universal enveloping algebra of $\g^*$ completed at the augmentation ideal. The coalgebra $\Ucon(\g)$ can be trivial, i.e. isomorphic to $\ground$; such is the case when $\g$ is semisimple. In contrast, if $\g$ is conilpotent, then $\g^*$ embeds into $[\U_{\operatorname{con}}(\g)]^*$ this is well-known if $\g$ is finite-dimensional (and so $\g^*$ is nilpotent) and is easy to prove in general.
\end{rem}
Given a $\Ucon(\g)$-comodule $M$ (necessarily conilpotent), the corestriction determines a structure of a conilpotent $\Uc(\g)$-comodule on $M$ and thus, that of a conilpotent $\g$-comodule on $M$. Conversely, by universal properties of $\Uc(\g)$ and $\Ucon(\g)$, every conilpotent $\g$-comodule corresponds to a $\Ucon(\g)$-comodule. We obtain the following result.
\begin{prop}\label{prop:conilpotentmodules}
	The categories of conilpotent $\g$-comodules and $\Ucon(\g)$-comodules are equivalent.
\end{prop}
So, to a vector space $V$, we associated two Lie coalgebras: $\Lclf(V)$, the cofree locally finite  Lie coalgebra and $\Lcon(V)$, the cofree conilpotent Lie coalgebra. They have universal enveloping coalgebras $\Uc(\Lclf(V))\cong \Tc(V)$ and $\Ucon(\Lcon(V))\cong \Tcon(V)$, which are the cofree coassociative coalgebra and cofree conilpotent coassociative coalgebra respectively. We will need another coassociative coalgebra canonically associated to $V$.
\begin{defi}
The \emph{restricted cofree coalgebra}, on a vector space $V$ is defined as  $$\Tcres(V):=\Uc(\Lcon(V)),$$ the universal enveloping coalgebra of $\Lcon(V)$.
\end{defi}
Note that if $\dim V=1$, then $\Tcres(V)\cong \Tc(V)$; however if $V$ is concentrated in degree zero and has dimension at least 2,
then $\Tcres(V)$ is a proper subcoalgebra of $\Tc(V)$. We have, therefore, the following inclusions of coalgebras
$\Tc(V)\supseteq \Tcres(V)\supseteq \Tcon(V)$.

Just as $\Tcon(V)$  and $\Tc(V)$, the coalgebra $\Tcres(V)$ satisfies an appropriate universal property. Let us call a coassociative coalgebra $C$ \emph{Lie-conilpotent} if its associated Lie coalgebra $L(C)$ is conilpotent. Clearly, any conilpotent coassociative coalgebra is Lie-conilpotent, however there exist non-conilpotent coassociative coalgebras that are Lie-conilpotent, with $\Tcres(V)$ being the universal example. The following result follows from unwrapping the definitions.
\begin{prop}\label{prop:Lieconilpotent} The functor $V\mapsto \Tcres V$ from vector spaces to Lie-conilpotent coassociative coalgebras is right adjoint to the forgetful functor.
\end{prop}

Given a pc vector space $U$, denote by $\Tres(U)$ the pc algebra dual to $\Tcres(U^*)$. It has an obvious universal property dual to that of $\Tcres(U^*)$ with respect to maps to \emph{Lie-nilpotent} pc algebras (i.e. those whose associated pc Lie algebras are pronilpotent). For later use, we will establish some homological properties of $\Tres(U)$ similar to those of $\check{T}(U)$, cf. \cite[Proposition 2.4]{GL}.
\begin{lem}\label{lem:Lieconil}
	For any pc vector space $U$, there is a $\Tres(U)$-bimodule resolution of the form
	\begin{equation}\label{eq:shortresolution}
	\Tres(U)\otimes U\otimes\Tres(U)\stackrel{d}{\longrightarrow} \Tres(U)\otimes\Tres(U)\stackrel{m}{\longrightarrow}\Tres(U).
	\end{equation}
	where $m$ is the multiplication map and $d(1\otimes u\otimes 1 )=u\otimes 1-1\otimes u$ for $u\in U$.
\end{lem}	
\begin{proof}
	Note that for any (pc or not) algebra $A$, the kernel of the multiplication map $A\otimes A\to A$ is $\Omega_A$, the $A$-bimodule of noncommutative differentials for $A$, and it has the property that any derivation of $A$ with values in an $A$-bimodule $M$ is in a 1-1 correspondence with $A$-bimodule maps $\Omega(A)\to M$. Combined with the universal property of $\Tres(U)$, this implies that $$\Omega_{\Tres(U)}\cong \Tres(U)\otimes U\otimes\Tres(U).$$ That the map $d:\Omega_{\Tres(U)}\to \Tres(U)\otimes \Tres(U)$ is as claimed follows e.g. from the comparison of (\ref{eq:shortresolution}) with a well-known two-term bimodule resolution of the ordinary tensor algebra of $U$ (viewed as a discrete vector space) over itself.
\end{proof}	
\begin{cor}\label{cor:resol}
	Given a pc module $M$ over $\Tres(U)$, there is the following two-term resolution of $M$:
	\[
	\Tres(U)\otimes U\otimes M{\longrightarrow} \Tres(U)\otimes M{\longrightarrow}M.
	\]
	where the righthand arrow is the action of $\Tres(U)$ on $M$ and the lefthand arrow has the form $1 \otimes u \otimes m \mapsto u\otimes m-1\otimes um$.
\end{cor}
\begin{proof}
	This is obtained by tensoring (\ref{eq:shortresolution}) with $M$ over $\Tres(U)$.
\end{proof}	
\subsection{PBW theorem}
Recall that in the ordinary setting of Lie algebras and their universal enveloping algebras, we have an injection $\h\to \U(\h)$ from a Lie algebra $\h$ into its universal enveloping algebra. There is, furthermore, a filtration on $\U(\h)$ where $F_p(\U(\h))$ is formed by linear combinations of monomials in $\h$ of length no more than $p$.  The PBW theorem states that the associated graded to $\U(\h)$ is isomorphic as a Hopf algebra to $S(\h)$, the symmetric algebra on $\h$. There is an analogue of this statement for universal enveloping coalgebras due to Michaelis \cite{Mic} which states that for a Lie coalgebra $\g$, the graded vector spaces $\operatorname{gr} \Uc(\g)$ and $\operatorname{gr}\Sc(\g)$ are isomorphic as Hopf algebras. Here $\Sc(\g)$ stands for the cofree cocommutative coalgebra on the vector space  $\g$ and it is rather bigger than the ordinary symmetric (co)algebra. It is noted in op. cit. that without taking $\operatorname{gr}(\Sc)$ the statement is not true.

The statement of the dual PBW theorem was further strengthened in \cite[Theorem 4.9]{Block} where it was proved that $\operatorname{gr} \U(\g)$ is, in fact, isomorphic to $\Scon(\g)$, the conilpotent symmetric coalgebra on $\g$; note that $\Scon(\g)$ is rather smaller than $\Sc(\g)$.
This seemingly paradoxical answer can be explained as follows (assuming for simplicity that $\g$ is finite-dimensional). The PBW filtration  on $\Uc(\g)$
induces one on $(\Uc(\g))^*\cong \check{\U}(\g^*)$. This latter filtration is just the PBW filtration on $\U(\g^*)\subset \check{\U}(\g^*)$. It is topologically exhaustive (meaning that $\bigcup_p F_p\U(\g^*)=\U(\g^*)$ is dense in $\check{\U}(\g^*)$), however it is not exhaustive since $\U(\g^*)\neq \check{\U}(\g^*)$. The usual PBW theorem then gives that $\operatorname{gr}\check{\U}(\g^*)\cong S(\g^*)$, the ordinary symmetric algebra on $\g^*$ and the dual statement about $\operatorname{gr} \Uc(\g)$ follows.

Without the dualization, the same phenomenon corresponds to the fact that the PBW filtration of $\Uc(\g)$ is not complete. Because of this, the PBW theorem for Lie coalgebras is of limited utility for homological considerations; we will not use it in this paper.

\section{Koszul duality I: conilpotent Lie coalgebras}\label{section:Koszul1}
Given a  Lie coalgebra $\g$ and a $\g$-comodule $M$, one can form its Chevalley-Eilenberg complex $\CE(\g,M)$ whose underlying  vector space is $S\Sigma^{-1}\g\otimes M$ and the differential is induced by the internal differential in $\g$, cobracket in $\g$ and the coaction of $M$. We denote by $\CE(\g)$ the  algebra $\CE(\g,\ground)$; then $\CE(\g,M)$ is a  $\CE(\g)$-module. It seems that not much is known about this complex; in particular whether it can be viewed as computing a derived functor in an appropriate model category of comodules. An attempt to give an answer to this question leads one to the Koszul duality between $\g$-comodules and $\CE(\g)$-modules similar to the Koszul duality between comodules over a coalgebra and modules over its cobar construction, cf. \cite{Posi}. However, comodules over Lie coalgebras are not necessarily locally finite (unlike comodules over coassociative coalgebras) and this adds an additional subtlety to the theory.

\subsection{Bar-cobar adjunctions}
Let $\g$ be a locally finite  Lie coalgebra and $M$ be a locally finite $\g$-comodule.
Let $\{e_i\}$ be a basis in $\g$ and $\{e^i\}$ the dual basis in $\g^*$. Then the canonical element $$\xi:=\sum_ie_i\otimes e^i\in \CE(\g)\hat{\otimes} \g^*$$ is MC and, abusing the notation, we will  denote by the same letter the element in $\CE(\g)\hat{\otimes}\check{\U}(\g^*)$ corresponding to $\xi$ under the canonical inclusion
$\g^*\to \check{\U}(\g^*)$.

Now consider the  vector space  $\CE(\g){\otimes} M$. It is a $\CE(\g)\hat{\otimes}\check{\U}(\g^*)$-module and we can twist it by the element $\xi$.
\begin{defi}
	The  $\CE(\g)$-module $\CE(\g, M):=(\CE(\g)\otimes M)^{[\xi]}$ is called the Chevalley-Eilenberg complex of $\g$ with values in $M$.
\end{defi}
The functor $M\mapsto\CE(\g,M)$ is one part of the bar-cobar adjunction (namely it is the cobar construction of $M$). Given a $\CE(\g)$-module $N$, let us consider $N^*\hat{\otimes}  \check{\U}(\g^*)$. This is a $\CE(\g)\hat{\otimes}\check{\U}(\g^*)$-module and we can twist it by the element $\xi$. The resulting pc $\check{\U}(\g^*)$-module  $N^*\hat{\otimes} \check{\U}(\g^*)^{[\xi]}$ corresponds to a $\g$-comodule with the underlying space $N\otimes \Uc(\g)$ called the bar construction $B(N)$ of the $\CE(\g)$-module $N$.
\begin{theorem}
	The functors $\CE(\g,-)$ and $B(-)$ form an adjoint pair between the categories of $\CE(\g)$-modules and of locally finite $\g$-comodules.
\end{theorem}
\begin{proof}
Let $M$ be a locally finite $\g$-comodule and $L$ be a $\CE(\g)$-module. Then we can deduce that $\Hom_{\CE(\g)\MMod}(\CE(\g,M),L)$ has $\Hom(M,L)\cong M^*\hat{\otimes}L$ as its underlying graded vector space. It is then a $\check{\U}(\g^*)\hat{\otimes}\CE(\g)$-module and the differential in it
agrees with the differential in
$(M^*\hat{\otimes}L)^{[\xi]}$.

Similarly, the vector space $\Hom_{\g\CComod}(M,B(L))$ is isomorphic to $(M^*\hat{\otimes}L)^{[\xi]}$.
\end{proof}	
It is, therefore, natural to ask whether the above bar-cobar adjunction can be promoted to a Quillen equivalence. The following example shows that it is unlikely without restrictions on $\g$ beyond local finiteness.

\begin{example}\label{ex:sl2}
	Let $\g$ be the Lie coalgebra dual to $sl_2(\ground)$. Let us assume that $\ground$ is algebraically closed. Then $\CE(\g)$ is an algebra, quasi-isomorphic to $\Lambda(x)$, the exterior algebra with one generator in degree $3$. The derived category of $\CE(\g)$ is equivalent to that of $\Lambda(x)$ and the latter is
	not semisimple. Thus, any of the finer coderived categories of $\CE(\g)$
	are also not semisimple.
	On the other hand, the category of locally finite $sl_2(\ground)$-modules is semisimple and so the coderived category of $\Uc(\g)$ is likewise semisimple. 	
\end{example}

\subsection{Acyclic abelian Lie coalgebras} As a warm-up, we will compute the coderived categories for some of the simplest types of conilpotent Lie coalgebras, namely those which are abelian and have an acyclic differential. These results will also be needed later on.

An abelian acyclic Lie coalgebra is just an acyclic complex of $\ground$-vector spaces and as such, it is decomposed into a direct sum of elementary acyclic complexes of length 2 with basis vectors $x$ and $dx$. We will distinguish three types of such complexes:
\begin{enumerate}
	\item The degree of $x$ is $-1$. The corresponding (abelian) Lie coalgebra will be denoted by $\g_1$.
\item 	The degree of $x$ is $0$. The corresponding (abelian) Lie coalgebra will be denoted by $\g_2$.
\item The degree of $x$ is different from $0$ and $-1$. The corresponding (abelian) Lie coalgebra will be denoted by $\g_3$.
\end{enumerate}
\begin{prop}\label{prop:coderivedabelian} Assume that $\ground$ is algebraically closed. All three coderived categories above are semisimple. Moreover:\begin{enumerate}\item
	The coderived category of $\g_1$ has one simple object $\ground_\alpha$ for any $\alpha\in\ground$, up to shift. The endomorphism algebra of every simple object is $\ground$ and there are no non-zero morphisms between different simple objects.
	
\item	The coderived categories of $\g_2$ and $\g_3$ each have only one simple object up to shift and $\ground$ worth of its endomorphisms.
\end{enumerate}
\end{prop}
\begin{proof}
	The case of $\g_3$ is the simplest: the universal enveloping coalgebra of $\g_3$ is conilpotent and its dual is the pc algebra $\ground[[a,da]]$ for some generator $a$. The latter algebra is filtered quasi-isomorphic to $\ground$ (with respect to the filtrations by the powers of the maximal ideal) and so the coderived category of $\Uc(\g_3)$   is equivalent to the category of graded vector spaces.
	
	Consider the linear dual of the universal enveloping coalgebra of $\g_2$, that is the pc completion of the ordinary enveloping algebra of $\g_2^*$. Denoting by $y$ and $z$ the dual basis vectors to $x$ and $dx$ respectively; we see that the enveloping algebra of $\g^*_2$ is $\ground[y,u]$ with $|y|=-1, |u|=0$ and $dy=u$. Its pc completion is isomorphic to $\prod_{\alpha\in\ground}\ground[[y_\alpha,u_\alpha]]$ with $d(y_\alpha)=u_{\alpha}-\alpha$. It is therefore a coproduct of conilpotent coalgebras, and so every $\g_2$-comodule is isomorphic to a coproduct of comodules over coalgebras dual to pc algebras $\ground[[y_\alpha,u_\alpha]]$. One of such coalgebras, namely the one corresponding to $\alpha=0$, is weakly equivalent as a conilpotent coalgebra to $\ground$ since it is isomorphic to $\Uc(\g_3)$ for a suitable $\g_3$; therefore its coderived category is equivalent to that of $\ground$. For $\alpha\neq 0$, the pc algebra $\ground[[y_\alpha,u_\alpha]]$ has the property that its multiplicative identity is exact; this implies that any twisted module over it is chain homotopic to zero.   It follows that the coderived category of $\g_2$ is equivalent to that of $\ground[[y_0,dy_0]]^*$, which is the category of graded vector spaces.
	
	Finally consider $\g_1$. Its linear dual of the universal enveloping coalgebra is the pc completion of the ordinary enveloping algebra of $\g_1^*$ and is isomorphic to $\prod_{\alpha\in\ground} \ground[[z_\alpha,dz_\alpha]]$. Here $|z_\alpha|=0$. Each of the pc algebras $\ground[[z_\alpha,dz_\alpha]]$ is dual to the conilpotent coalgebra whose coderived category is (as we saw before) the same as that of $\ground$. It follows that each such conilpotent coalgebra has the same coderived category as $\ground$ so the coderived category of $\g_1$ is equivalent to that of the cosemisimple coalgebra $\coprod_{\ground_\alpha},\alpha\in \ground$. The required conclusion follows.
\end{proof}

\subsection{Koszul complex for symmetric coalgebras}\label{subsection:Koszulcoalgebras}
Koszul complexes are classically used to construct resolutions of modules of Koszul quadratic algebras, such as $\operatorname{S}(V)$, the symmetric algebra on a graded vector space $V$. In this subsection we develop a similar theory for comodules over the cofree cocommutative  coalgebra $\Sc(V)$.

\begin{prop}\label{prop:Koszulresolution}
	Let $\g$ be a (dg) vector space viewed as an abelian Lie coalgebra and $\Uc(\g)$ be its universal enveloping coalgebra. Then for any locally finite $\g$-comodule $M$, the unit of the adjunction
	\begin{equation}\label{eq:unit} M\to B\circ\CE(\g,M)\end{equation} is a weak equivalence as a map between $\Uc(\g)$-comodules.
\end{prop}
\begin{rem}
	Clearly the bar construction of any $\CE(\g,M)$-module is a fibrant $\Uc(\g)$-comodule; therefore Proposition \ref{prop:Koszulresolution} supplies a canonical fibrant resolution of any $\Sc(V)$-comodule where $V$ is a vector space. It is, therefore, a comodule version of the well-known Koszul resolution of a module over a symmetric algebra.
\end{rem}
Since $\g$ is abelian, it decomposes as a direct sum of one-dimensional abelian Lie coalgebras and the two-dimensional acyclic abelian Lie coalgebras of Proposition \ref{prop:coderivedabelian}. As a first step, we will prove the claim of Proposition \ref{prop:Koszulresolution} for each type separately.
\begin{lem}\label{lem:partialresult}
	The claim of Proposition \ref{prop:Koszulresolution} holds for one-dimensional abelian Lie coalgebras and for the three types of acyclic Lie coalgebras of Proposition \ref{prop:coderivedabelian}.
\end{lem}
\begin{proof} All of our constructions, including the adjoint pair $\CE(\g,-)$ and $B(-)$, commute with scalar extension from $\ground$ to its algebraic closure.
	Since a (dg) $\ground$-vector space is acyclic if and only if its scalar extension is,
	we may assume without loss of generality that $\ground$ is algebraically closed.

Suppose that $\g$ is one-dimensional and odd. In that case $\Uc(\g)$ is an exterior coalgebra on one generator, so it is conilpotent and the  map (\ref{eq:unit})  is just the ordinary cobar-bar resolution of a comodule over a conilpotent coalgebra, cf. \cite[Theorem 6.4]{Posi}.

If $\g$ is one-dimensional, even and \emph{not} concentrated in degree zero, then  $\Uc(\g)$ is a conilpotent symmetric algebra on one generator and is the bar construction of an exterior algebra on one generator. Again, the conilpotent Koszul duality for a bar construction, cf. \cite[Theorem 6.3]{Posi} gives that (\ref{eq:unit})  is a weak equivalence.

If $\g$ is one-dimensional sitting in degree zero, then $\Uc(\g)\cong\Sc(\g)$, the cofree cocommutative coalgebra on $\g$; note that it is \emph{not} conilpotent. In that case $\Uc(\g)$ is the \emph{extended} bar construction of an exterior algebra on one generator in degree 1 and the nonconilpotent Koszul duality \cite[Theorem 3.10 (2)]{GL} for a bar construction gives the claimed result.

Now suppose that $\g$ is a two-dimensional acyclic Lie coalgebra. To show that (\ref{eq:unit}) is a weak equivalence, it suffices to show that
for any $\g$-comodule $N$, there is an isomorphism $H^*[\Hom_{\Uc(\g)}(N, B\circ\CE(\g,M))]\cong [N,M]_*$ where $[N,M]_*$ stands for the graded vector space of maps $N\to M$ in the coderived category of $\g$-comodules. Note that \[\Hom_{\Uc(\g)}(N, B\circ\CE(\g,M))\cong \CE(\g, \Hom(N,M));\] this isomorphism holds for any Lie coalgebra $\g$, not necessarily acyclic or abelian. It suffices, therefore, to prove that there is an isomorphism
\begin{equation}\label{eq:quasii}
	[N,M]_*\cong\CE(\g, \Hom(N,M)).
\end{equation}

Now let $\g=\g_2$; then, according to Proposition \ref{prop:coderivedabelian} the coderived category of $\g$ is generated by one simple module $\ground$ (on which $\g$ coacts trivially). Therefore, it  suffices to assume that $M=N=\ground$, with the trivial coaction of $\g$. Clearly, $\CE(\g,\ground)\cong\CE(\g)$ has cohomology sitting in degree zero only and is isomorphic to $\ground$. The space of homotopy classes of maps $[\ground,\ground]_*$ is likewise isomorphic to $\ground$ and the map $\CE(\g)\cong [\ground,\ground]_*$ is clearly an isomorphism. The case $\g=\g_3$ is completely analogous.

The case $\g=\g_1$ is slightly different. Note that $\CE(\g_1)\cong \ground[z,dz]$, the polynomial de Rham algebra in one variable. Recall that the coderived category of $\g$-comodules is generated by simple objects $\ground_{\alpha}, \alpha\in\ground$ so that $[\ground_\alpha,\ground_\beta]_*=\begin{cases}\ground, \text{if~} \alpha=\beta,\\ 0, \text{otherwise.}\end{cases}$ Moreover, the $\g_1$-comodules $\ground_\alpha,\ground_\beta$ correspond to the MC elements $\alpha dz,\beta dz\in \ground[z,dz]$. Set $N=\ground_{\alpha}$ and $=M\ground_{\beta}$. Then $\Hom(N,M)\cong\ground_{\alpha-\beta}$ and $\CE(\g_1, \Hom(N,M))$ is computed by the twisted complex
$\ground[z,dz]^{[(\alpha-\beta) dz]}$ of $\ground[z,dz]$. The differential $d^{[\alpha-\beta]}$ in $\ground[z,dz]^{[(\alpha-\beta) dz,]}$ has the form
\[
d^{[\alpha,\beta]}(\omega)=d\omega+(\beta dz-\alpha dz)\omega
\]
for $\omega\in\ground[z,dz]$.

The zeroth cohomology of $\ground[z,dz]^{[(\alpha-\beta) dz]}$ correspond to polynomials $f$ satisfying the differential equation
\[
f'+(\alpha-\beta)f=0
\]
and this equation only has polynomial solutions when $\alpha=\beta$ and in that case the space of solutions is one-dimensional.
The first cohomology corresponds to the quotient space of all polynomials $q(z)$ modulo those of the form $p'+(\alpha-\beta)p$.
The differential equation $p'(z)+(\alpha-\beta)p(z)=q(z)$ always has a polynomial solution and we conclude that the first cohomology group is zero.

So the vector space $\CE(\g_1, \Hom(\ground_\alpha,\ground_\beta))$ is quasi-isomorphic to $\ground$ for $\alpha=\beta$ and to zero otherwise. It follows that the map (\ref{eq:quasii}) is an isomorphism and so the desired claim is proved for $\g_1$.
\end{proof}

\begin{lem}\label{lem:tensorprod}
 Let $f_\alpha : M_\alpha \to N_\alpha$ be a family of weak equivalences of comodules over coalgebras $C_\alpha,\alpha\in S$. Then
\[ \bigotimes_{\alpha\in S} f_\alpha: \bigotimes_{\alpha\in S}M_\alpha\to \bigotimes_{\alpha\in S}N_\alpha.\]
 is a weak equivalence of $\bigotimes_{\alpha\in S} C_{\alpha}$-comodules.
	\end{lem}
\begin{proof}
Suppose first that $S$ consists of two elements. The condition that $f:M_1\to N_1$ is a weak equivalence means that the cofiber $C(f)$ of $f$ is coacyclic as a $C_1$-comodule. The cofiber of the map $M_1\otimes M_2\to N_1\otimes M_2$ has the form $C(f)\otimes M_2$ and it is clearly coacyclic as a $C_1\otimes C_2$-comodule. Thus, the $C_1\otimes C_2$-comodules $M_1\otimes M_2$ and $N_1\otimes M_2$ are weakly equivalent. Similarly $N_1\otimes M_2$ and $N_1\otimes N_2$ are weakly equivalent, so the claim is proved in this case and then, by induction, for any finite set $S$. The general result follows by transfinite induction.
\end{proof}
We now return to the proof of 	 Proposition \ref{prop:Koszulresolution}. While it relies on a result on solvable Lie coalgebras belonging to the following subsection, the reader can check that it does not cause the circularity of the argument.
\begin{proof}[Proof of Proposition \ref{prop:Koszulresolution}]
	Any $\Uc(\g)$-comodule is a filtered colimit (union) of its finite-dimensional subcomodules (as is the case for comodules over any coalgebra). A filtered colimit of weak equivalences is a weak equivalence and so it suffices to assume that $M$ is finite-dimensional.
	
	 Note that $\Uc(\g)$ is isomorphic to the tensor product of coalgebras of the form $\Uc(\h)$ where $\h$ is one of the coalgebras of Lemma \ref{lem:partialresult}. Let us write it as $\Uc(\g)\cong\bigotimes_{\alpha\in S}\Uc(\h_\alpha)$ where $S$ is some indexing set.
	Since $\g$ is abelian, $M$ belongs to the triangulated subcategory of the coderived category of $\Uc(\g)$ generated by one-dimensional comodules, by virtue of Corollary~\ref{cor2:filtration}. Hence we may assume that $M$ is one-dimensional. 	Thus
	$M\cong \bigotimes_{\alpha\in S} M_{\alpha}$, for some one-dimensional $\Uc(\h_\alpha)$-comodules $M_\alpha$, and the unit of the adjunction
	  $\eta:M\to B\circ\CE(\g,M)$
	  is isomorphic to the tensor product of the units
	  $\eta_\alpha: M_\alpha\to B\circ\CE(\g_\alpha,M_\alpha)$.
	  By Lemma~\ref{lem:partialresult}, each $\eta_\alpha$ is a weak equivalence of  $\Uc(\g_\alpha)$-comodules, and it follows by Lemma~\ref{lem:tensorprod} that $\eta$ is a weak equivalence  of $\Uc(\g)$-comodules.
	\end{proof}

\subsection{Koszul duality for locally finite comodules} Recall from \cite{GL} that the category $A$-$\Mod$ of modules over any  algebra  $A$ has the structure of a compactly generated model category of second kind where weak equivalences are the maps $L\to N$ which induce quasi-isomorphisms $\Hom_A(M,L)\to\Hom_A(M,N)$ for any finitely generated twisted $A$-module $M$ and fibrations are surjective maps. This model structure is different, in general, from the standard one; specifically a quasi-isomorphism of $A$-modules is not necessarily a weak equivalence in the above sense.  Denote by $\PerfII(A)$ the full subcategory of $A$-modules formed by finitely generated twisted $A$-modules and their retracts up to homotopy; these represent compact objects in the corresponding homotopy category of $A$-$\Mod$, the compactly generated coderived category of $A$.
For a conilpotent  Lie coalgebra $\g$, we consider this model structure on $\CE(\g)$-modules. Note that, owing to the conilpotency of $\g$, $\CE(\g)$ is cofibrant as a \emph{commutative}  algebra but not as an associative  algebra.

We need some results about the structure of (dg) modules over finite-dimensional solvable Lie algebras. Recall that, under the assumption that $\ground$ is algebraically closed, an ordinary (ungraded) finite-dimensional module over a solvable finite-dimensional Lie algebra $\g$ always possesses an invariant subspace of dimension 1. It follows that such a module always has a filtration whose associated graded module is a direct sum of 1-dimensional representations of $\g$. In the dg situation, the result is slightly more complicated.
\begin{lem}\label{lem:nilpotentfiltration}
Let $\g$ be a (dg) finite-dimensional solvable  Lie algebra over an algebraically closed field $\ground$ and $N$ be a finite-dimensional (dg) $\g$-module that has no proper (dg) $\g$-submodules. Then $N$ is either $1$-dimensional or $2$-dimensional. If $N$ is $2$-dimensional, then it is spanned by two vectors $x$ and $dx$ and the $\g$-module structure on $N$ is given by the formulas:
\[
a(x)=\alpha(a)x, \quad a(dx)=\alpha(a)dx,
\]
for $a\in \g$, $|a|=0$ where $\alpha$ is a linear function on $\g^0$ and
\[
a(x)=0, \quad a(dx)=\alpha(da)x,
\]
for $a\in\g$ and $|a|=-1$.
If $a\in\g$ has degree different from $0$ or $-1$, then $a$ acts as zero on $N$.
\end{lem}	
\begin{proof}
Since $\g$ is solvable, we know that $N$ has a 1-dimensional $\g$-invariant submodule; let $x$ be its basis vector. Note that any element in $\g$ of degree different from zero must act on $x$ trivially for degree reasons. Therefore, every homogeneous component of $x$ itself spans an invariant $\g$-submodule and so we assume, without loss of generality, that $x$ is homogeneous. Then for any $a\in \g$ with $|a|=0$, we have $ax=\alpha(a)x$ for some linear function $\alpha$ on $\g^0$ and $ax=0$ if $|a|\neq 0$. Next,
\[d(ax)=da(x)+(-1)^{|a|}adx.\]
If $|a|=0$, necessarily $da(x)=0$ and $a(dx)=\alpha(a)dx$. If $|a|=-1$, then $ax=0$ and
\[da(x)=\alpha(da)x=a(dx)\]
as claimed. A similar reasoning shows that if $|a|\neq 0,1$ then $ax=a(dx)=0$. 		
\end{proof}		

The next result shows that the existence of this exotic 2-dimensional $\g$-module does not matter, after all, as it represents zero in the coderived category of (locally finite) $\g^*$-comodules.
\begin{lem}\label{lem:coacyclic}
Let $\g$ be a finite-dimensional (dg) Lie coalgebra over an algebraically closed field~$\ground$ whose dual $\g^*$ is solvable and let $N$ be a $2$-dimensional $\g^*$-module constructed in Lemma \ref{lem:nilpotentfiltration}. Then $N$ is coacyclic as a $\g$-comodule.
\end{lem}	
\begin{proof}
Consider an abelian Lie algebra spanned by two vectors $a$ and $da$ with $|a|=-1$ (so it's acyclic) and acting on $N$ by the formulas:
\[
(da)(x)=\alpha x, (da)(dx)=\alpha(dx);
a(x)=0, a(dx)=\alpha x
\]
where $\alpha\in\ground$.	Note that the dual to this abelian Lie algebra is the (abelian) Lie coalgebra $\g_2$ of Proposition \ref{prop:coderivedabelian}.  Then it is clear that $\g^*$ acts on $N^*$  through a Lie algebra map $\g^*\to\g_2^*$ and so, if $N^*$ is coacyclic as an $\g_2$-comodule, it is also coacyclic as an $\g$-comodule.

From the description of $\g_2$-comodules in Proposition \ref{prop:coderivedabelian} it follows that all $\g_2$-comodules that are nontrivial in the coderived category of $\g_2$-comodules, have nontrivial cohomology. Since $N^*$, and therefore also $N$, is acyclic, it represents the zero object in the coderived category and so, is coacyclic.
\end{proof}	

\begin{cor}\label{cor:filtration}
	Let $\g$ be a finite-dimensional (dg) Lie coalgebra  over an algebraically closed field $\ground$ whose dual $\g^*$ is solvable and $M$ be a (dg) finite-dimensional $\g$-comodule. Then  $M$ is weakly equivalent to a comodule having a finite filtration whose subquotients are either one-dimensional or coacyclic.
\end{cor}
\begin{proof} By Lemma \ref{lem:coacyclic}, the comodule $M$ must have a quotient $L$ that is either two-dimensional coacyclic or one-dimensional. Then the kernel $M_1$ of the map $M\to L$ has the dimension less than $\dim M$  and also has a similar quotient. Continuing this process, we obtain a sequence of $\g$-comodules $0\hookrightarrow M_n\hookrightarrow M_{n-1}\hookrightarrow\ldots\hookrightarrow M_0=M$
such that each quotient $M_i/M_{i-1}$ has the required form.
\end{proof}	
Let us call a Lie coalgebra \emph{cosolvable} if it is a union of finite-dimensional Lie coalgebras whose duals are solvable Lie algebras. Then, Corollary \ref{cor:filtration} has the following straightforward generalization:
\begin{cor}\label{cor2:filtration}
	Let $\g$ be a cosolvable (e.g. conilpotent) Lie coalgebra over an algebraically closed field $\ground$ and $M$ be a locally finite $\g$-comodule. Then $M$ is weakly equivalent to a comodule having an increasing filtration by subcomodules whose subquotients are either coacyclic or $1$-dimensional.
\end{cor}

\begin{theorem}\label{thm:coreflective}
Let $\g$ be a conilpotent Lie coalgebra.	The adjoint pair $B(-)$ and $\CE(\g,-)$ determine a Quillen adjunction between model categories $\Uc(\g)\CComod$ and $\CE(\g)$-$\Mod$; moreover, the composition $B\circ \CE(\g,-)$ is an endofunctor on $\Uc(\g)\CComod$ which induces a functor isomorphic to the identity on the homotopy category of $\Uc(\g)\CComod$.

\end{theorem}
\begin{proof}
	As in the proof of Proposition~\ref{prop:Koszulresolution}, we may assume, without loss of generality, that $\ground$ is algebraically closed.
	It is immediate that the left adjoint functor $\CE(\g,-)$ respects cofibrations while the right adjoint $B(-)$ respects fibrations, so the pair $(B(-),\CE(\g,-))$ is indeed a Quillen adjunction.  Let $L$ be a locally finite $\g$-comodule, or, equivalently, a $\Uc(\g)$-comodule and consider the natural map $L\to B\CE(\g,L)$. To prove that this is a weak equivalence of $\Uc(\g)$-comodules it suffices to show that for any finite dimensional $\Uc(\g)$-comodule $M$, it induces a bijection $[M,L]_*\to [M,B\CE(\g,L)]_*$ where $[-,-]_*$ stands for the graded vector space of maps in the homotopy category of $\Uc(\g)$-comodules. The latter space of maps is computed by resolving the second argument by a fibrant comodule. Since $B\CE(\g,L)$ is already fibrant, we have an isomorphism
	 $[M,B\CE(\g,L)]_*\cong \CE(\g, \Hom(M,L))$. On the other hand, taking the standard cobar-bar resolution of the $\Uc(\g)$-comodule $L$, we obtain that $[M,L]_*$ is computed by the complex whose underlying graded vector space is $\Omega(\Uc(\g))\otimes \Hom(M,L)$ and a twisted differential making it the standard complex computing $\operatorname{HH}(\Uc(\g),\Hom(M,L))$, the Hochschild cohomology of the coalgebra $\Uc(\g)$ with coefficients in the bicomodule $\Hom(M,L)$. This Hochschild cohomology computes $\RHom_{\Uc(\g)\otimes (\Uc(\g))^{\op}}(\Hom(M,L),\Uc(\g))$  and therefore, by Proposition \ref{prop:onetwosided}, $\RHom_{\Uc(\g)}([\Hom(M,L)]^{\ad},\ground)$. The latter, in turn is isomorphic to $$\Omega(\Hom(M,L)):=[\Omega(\Uc(\g))\otimes [\Hom(M,L)]^{\ad}]^{[\xi]},$$ the cobar-construction of the $\Uc(\g)$-comodule $[\Hom(M,L)]^{\ad}$.
	
	There is a map of algebras $\Omega(\Uc(\g))\to\CE(\g)$ given by the canonical MC element in $\Hom(\Uc(\g),\CE(\g))$ and this determines a map $\Omega(N):=[\Omega(\Uc(\g))\otimes N]^{\xi}\to\CE(\g, N)$ for any  locally finite $\g$-comodule $N$. We claim that this latter map is a quasi-isomorphism; taking for $N$ the $\g$-comodule $[\Hom(M,L)]^{\ad}$ will give the desired statement.
	
	To prove that, recall from Corollary \ref{cor2:filtration} that $N$ possesses a filtration whose subquotients are one-dimensional $\g$-modules or two-dimensional coacyclic ones. The filtration on $N$ induces filtrations on $\Omega(N)$ and on $\CE(\g, N)$; the map $\Omega(N)\to\CE(\g, N)$ respects this filtration and if it is a quasi-isomorphism on the associated subquotients, it will be a quasi-isomorphism overall. This reduces the desired statement to the case when $N$ is 1-dimensional or coacyclic. In the latter case both $\Omega(N)$ and $\CE(\g,N)$ have zero cohomology so we are left with the one-dimensional module $N$.
	
	To prove that $\Omega(N)\to\CE(\g, N)$ is a quasi-isomorphism for a 1-dimensional $N$, note that $\Omega(N)$ has the same underlying graded vector space as $\Omega(\Uc(\g))$ but the differential has the form $d_{\Omega}+\xi_{N}$ where $\xi_N$ is the MC element in $\Omega(\Uc(\g))$  corresponding to the $\g$-comodule $N$. Similarly, the differential on $\CE(\g,N)$ has the form $d_{\Omega}+\xi_N$ where we denote (slightly abusing the notation) the MC element in $\CE(\g,N)$ corresponding to the $\g$-comodule $N$ by the same symbol $\xi_N$.
	
	Now consider the canonical increasing filtration $0\subset\g_0\subset \g_1\subset\ldots$ on the conilpotent Lie coalgebra $\g$  whose dual is the lower series filtration on  $\g^*$: $\g^*\supset [\g^*,\g^*]\supset [[\g^*,\g^*],\g^*]\supset\ldots$. In particular, the lowest filtration component $\g_0$ consists of those elements that vanish on all commutators of $\g^*$. Note that the associated graded Lie coalgebra is abelian. This induces an exhaustive filtration on $\CE(\g)$ and on $\Omega(\Uc(\g))$.
	
	Since $N$ is one-dimensional, all commutators of $\g^*$ vanish on it; this translates into $\xi$ being in the lowest filtration component of $\CE(\g)$ and similarly for $\Omega(N)$. Thus, the total differential on $\CE(\g,N)$ and on $\Omega(N)$  preserve the given filtrations.  Next, $d_{\CE}$ vanishes on the associated graded module to $\CE(\g,N)$ and so, $\operatorname{gr}\CE(\g,N)$ is the Chevalley-Eilenberg complex of the comodule $\operatorname{gr}N  = N$ over the Lie coalgebra $\operatorname{gr} \g$, which is an abelian Lie coalgebra. Similarly, $\operatorname{gr}\Omega (N)$ is the cobar construction of $N$ as a $\operatorname{gr} \g$-comodule.
	Thus, showing that $\operatorname{gr}\CE(\g,N)\to\operatorname{gr}\Omega(N)$ is a quasi-isomorphism (which implies the desired statement) is equivalent to the map $\CE(\g,N)\to\Omega(N)$ being a quasi-isomorphism under the assumption that $\g$ is abelian.
	
	To prove this last statement, it suffices to observe that $\CE(\g, N)$ computes $[\ground,N]_*$ with the $\Uc(\g)$-comodule $N$ being replaced by its Koszul complex which, by Proposition \ref{prop:Koszulresolution}, is a fibrant replacement of $N$, whereas $\Omega(N)$ computes the same with the standard cobar-bar resolution of $N$. This finishes the proof.
\end{proof}

An immediate consequence of Theorem \ref{thm:coreflective} is an interpretation of the Chevalley-Eilenberg complex $\CE(\g,M)$ for a Lie coalgebra $\g$ and a $\g$-comodule $M$ as a derived functor in the abelian category of locally finite $\g$-comodules.
\begin{cor}\label{cor:derived}
	Let $\g$ be a conilpotent Lie coalgebra over a  field $\ground$ and $M$ be a locally finite $\g$-comodule. Then $\CE(\g,M)$ computes $\RHom(\ground, M)$ in the derived category of locally finite $\g$-comodules.
\end{cor}
\begin{proof}
	Indeed, $\RHom(\ground, M)$ is computed by $\Omega(M)$ but it was shown in the course of the proof of Theorem \ref{thm:coreflective} that $\Omega(M)$ and $\CE(\g,M)$ are quasi-isomorphic.
\end{proof}
\begin{rem} Theorem \ref{thm:coreflective} exhibits the category  $H^0(\g\CComod)$ as a coreflective subcategory of $H^0(\CE(\g)\MMod)$ via the functor $\CE(\g,-)$; in particular this implies that this functor is fully faithful on the level of homotopy categories.
	A natural question is whether $\CE(\g,-)$ is also essentially surjective up to homotopy, i.e. whether it is part of a Quillen equivalence between $\Uc(\g)\CComod$ and $\CE(\g)\MMod$. The following example shows that this is not true, in general.
\end{rem}

\begin{example}\label{ex:deRham} Let us assume that $\ground$ is algebraically closed.
	Let $\g_1$ be the two-dimensional  Lie coalgebra of Proposition \ref{prop:coderivedabelian}. Then, the algebra $\CE(\g)$ is isomorphic to the ordinary polynomial de Rham algebra in one variable $\CE(\g)\cong \ground[z,dz]$
	where $|z|=0$. This algebra is quasi-isomorphic to $\ground$ and so, its ordinary derived category (and thus, the conilpotent coderived category of $\g_1$) is equivalent to the category of graded $\ground$-vector spaces.
	On the other hand, the coderived category of $\ground[z,dz]$ is not trivial. Indeed, any polynomial $1$-form in $\ground[z,dz]$ is an
	MC-element, and it is easy to check that for two different such MC-elements $\omega_1$ and $\omega_2$ the twisted modules $\ground[z,dz]^{[\omega_1]}$ and $\ground[z,dz]^{[\omega_2]}$ have no
	nonzero maps between them (this follows from the fact that the differential equation $f\omega_1-\omega_2 f=df$ does not have polynomial solutions).
	
	Next,  Proposition \ref{prop:coderivedabelian} gives us that the coderived category of $\Uc(\g_1)$ is equivalent to that of an infinite sum of copies of $\ground$. The indecomposable 1-dimensional $\g_1$-comodule $\ground=\ground_{\alpha}, \alpha\in\ground$ corresponds to an MC element $\alpha dz\in \ground[z,dz]$ and thus, to a corresponding twisted $\ground[z,dz]$-module but clearly there are many inequivalent others (which can be viewed as $L_\infty$  $\g$-comodules). So, the cobar construction exhibits the coderived category $\Uc(\g_1)$ as a coreflective subcategory of the compactly generated coderived category of $\CE(\g_1)$, but it is not equivalent to it.
\end{example}
This example shows that, in order to obtain a Quillen equivalence between $\Uc(\g)\CComod$ and $\CE(\g)\MMod$, further restrictions need to be imposed on $\g$. A natural such restriction is for $\g$ to be non-negatively graded; note that the two-dimensional  Lie coalgebra of Example \ref{ex:deRham} is clearly not of this type. The class of non-negatively graded  Lie coalgebras includes ordinary Lie coalgebras as well as Lie coalgebras associated with rational nilpotent topological spaces.
\subsubsection{Comodules over non-negatively graded  Lie coalgebras}
We will now prove that the functor $\CE(\g,-)$ is essentially surjective (on the homotopy category level) in the case when the conilpotent Lie coalgebra $\g$ is non-negatively graded. To establish this,  we need to recall certain results on the Koszul duality (or triality) between modules over Lie algebras and co/contra-modules over their Chevalley-Eilenberg coalgebras.

Let $G$ be a  Lie algebra. Then its category of modules $G\MMod$ is equivalent to the category of $\U(G)$-modules and as such, carries a model category structure with weak equivalences being quasi-isomorphisms and fibrations being surjective maps. The Chevalley-Eilenberg coalgebra of $G$ is $\CE(G):=\Scon\Sigma G$, the cofree conilpotent cocommutative coalgebra on $\Sigma G$ with the differential induced by the internal differential of $G$ and the Lie bracket on $G$. Given a $G$-module $M$, the (cohomological) Chevalley-Eilenberg complex of $G$ with coefficients in $M$ has the form $[\hat{S}\Sigma^{-1}G^*\otimes M]^{\xi}$, the MC twisting of the $\hat{S}\Sigma^{-1}G^*$-module $\hat{S}\Sigma^{-1}G^*\otimes M$ similar to the one used in the definition of the Chevalley-Eilenberg complex of a Lie coalgebra. In fact, in the case when $G$ is finite-dimensional, $G^*$ is a Lie coalgebra, and the action map $G\otimes M\to M$ gives, by the tensor-hom adjunction a map
$M\to \Hom(G,M)\cong G^*\otimes M$ giving $M$ the structure of a $G^*$-comodule. Then we have an isomorphism $\CE(G,M)\cong \CE(G^*,M)$.

An important subtlety in this framework is that the cohomological Chevalley-Eilenberg complex $\CE(G,M)$ is a \emph{contramodule} over the coalgebra $\CE(G)$ rather than a comodule. In fact, there is also the corresponding $\CE(G)$-comodule, the \emph{homological} Chevalley-Eilenberg complex of $G$ having the form $\CE_*(G, M)\cong \Scon\Sigma G\otimes M$ with a suitable twisted differential. The $\CE(G)$-comodule $\CE(G, M)$ and the $\CE(G)$-contramodule $\CE_*(G, M)$ are related by the comodule-contramodule correspondence and we refer to \cite[Section 2,5]{Posi} for the definition of contramodules and the comodule-contramodule correspondence.

The category $C\CCtrmod$ of contramodules over any coalgebra $C$ admits a model structure, cf. \cite[Section 8.2]{Posi}, and the corresponding homotopy category is called the \emph{contraderived} category of $C$. The following result holds.

\begin{prop} Let $G$ be an arbitrary Lie algebra. The functor $\CE(G,-)$ is a left Quillen functor that is part of a Quillen equivalence between $G\MMod$ and $\CE(G)\CCtrmod$. Furthermore, the functor $\RHom_{\CE(G)}(\CE(G),-)$ is a left Quillen functor that is part of a Quillen equivalence between $\CE(G)\CComod$ and $\CE(G)\CCtrmod$. 
	\end{prop}
\begin{proof}
	The first claim is proved in \cite[Example 6.6]{Posi} in the special case when $G$ is an ordinary (non-dg) Lie algebra and the arguments extend verbatim to the dg case. The second claim is a general result on the correspondence between comodules and contramodules and is proved in \cite[Section 8.4]{Posi}.
\end{proof}	
 The full embedding of the homotopy category of $\Uc(\g)\CComod$ into that of $\CE(\g)$-$\Mod$ established in
Proposition~\ref{thm:coreflective},
can be improved to an equivalence, under a natural grading restriction.
\begin{theorem}\label{thm:nonneg}
	Let $\g$ be a non-negatively graded conilpotent Lie coalgebra.	The adjoint pair $B(-)$ and $\CE(\g,-)$ determines a Quillen equivalence between model categories $\Uc(\g)\CComod$ and $\CE(\g)\MMod$.
\end{theorem}
\begin{proof}
	We may assume, without loss of generality, that $\ground$ is algebraically closed.
	It suffices to prove that
	 $\CE(\g,-)$ is essentially surjective.
	Since the homotopy category of $\CE(\g)$-modules is generated by finitely generated twisted $\CE(\g)$-modules, it further suffices to show that any such is in the image of $\CE(\g,-)$. Let us assume, first of all, that $\dim \g<\infty$. A twisted $\CE(\g)$-module has the form $M=(\CE(\g)\otimes V)^{[\xi]}$ where $\xi$ is an MC element in $\CE(\g)\otimes \End(V)$ and $\dim V<\infty$. Note that, since $\g$ is non-negatively graded and finite-dimensional, $\CE(\g)$ is finite-dimensional in every graded component and $\CE(\g)\cong \CE(\g^*)$. So, $M$ can be viewed as a contramodule over the coalgebra $\CE_*(\g):=[\CE(\g^*)]^*$.
	
	The $\CE_*(\g)$-comodule corresponding to the contramodule $M$ via the comodule-contramodule correspondence, has the form $N:=\CE_*(\g)\otimes V$ (with a suitable twisted differential). Since $N$ is a cofree $\CE_*(\g)$-comodule and $\dim V<\infty$, it follows for any compact (i.e. finite-dimensional) $\CE_*(\g)$-comodule $L$, one has $\RHom_{\CE_*(\g)\CComod}(L,M)$ is finite-dimensional. Therefore, $M$ as an object in the contraderived category of $\CE_*(\g)$, also has this property as well as a $\g^*$-module $K$ corresponding to $M$ via the Koszul equivalence between $\g^*\MMod$ and $\operatorname{\CE_*(\g)-Ctrmod}$. Since $K$ is an object in the ordinary derived category of $\U(\g^*)$ and the latter is generated by a single compact object  $\U(\g^*)$, this property of the $\g^*$-module $K$ is equivalent to $H^*(K)$ being finite-dimensional. Next, since $\g$ is non-negatively graded, $\U(\g^*)$ is non-positively graded. Using the Postnikov tower of $K$ as a $\U(\g^*)$-module (which exists since $\U(\g^*)$ is non-positively graded), we see that $K$ belongs to the thick subcategory containing $\U(\g^*)$-modules that are finite-dimensional (and concentrated in a single degree). Every finite-dimensional $\U(\g^*)$-module is also a  $\U(\g)$-comodule. Thus, $M$ belongs to the thick subcategory generated by contramodules of the form $\CE(\g,L)$ where $L$ is finite-dimensional. By Theorem \ref{thm:coreflective} the functor $\CE(\g,-)$ is fully faithful  on the homotopy category level and it follows that, since each $\CE(\g,L)$ (by definition) is in the image of the functor $\CE(\g,-)$, so is $M$, as an object in the thick subcategory generated by $\CE(\g,L)$, for $L$ finite-dimensional. Thus, the desired claim is proved under the assumption $\dim\g<\infty$.
	
	Let now the conilpotent Lie coalgebra $\g$ be possibly infinite-dimensional. Then $\g=\bigcup_{\alpha}\g_\alpha$ where $\g_\alpha$ run through finite-dimensional Lie subcoalgebras of $\g$. Then $\CE(\g_\alpha)$ can be viewed as a subcomplex of $\CE(\g)$ and $\CE(\g)=\bigcup_\alpha\CE(\g_\alpha)$. Let $M$ be a twisted module over $\CE(\g)$ whose underlying space is $\CE(\g)\otimes V$ with $\dim V<\infty$. Because of finite-dimensionality of $V$, there exists $\h:=\g_\alpha$ such that the differential in $M$ restricts to $\CE(\h)\otimes V$. Denote by $N$ the corresponding twisted $\CE(\h)$-module whose underlying space is $\CE(\h)\otimes V$; the $\CE(\g)$-module $M$ is obtained from $N$ by tensoring up with $\CE(\g)$.
	
	Since $\dim\h<\infty$, there exists an $\h$-comodule $L$ such that $\CE(\h,L)$ is weakly equivalent to $N$. Since $\h$ is a Lie subalgebra of $\g$, $L$ can be viewed as a $\g$-module. We have:
	\[
	M\cong N\otimes_{\CE(\h)}\CE(\g)\simeq \CE(\h,L)\otimes_{\CE(\h)}\CE(\g)\cong \CE(\g,L)
	\]
	and we are done.
\end{proof}

\subsection{Conilpotent Koszul duality}
We will now discuss the coLie-commutative Koszul duality for \emph{conilpotent} modules over a Lie coalgebra. This version is closer to the ordinary
dg Koszul duality in that it does not involve model categories of modules of the second kind.

Let $\g$ be a Lie coalgebra and $\g\CComodcon$ be the category of conilpotent $\g$-modules. Recall from Proposition \ref{prop:conilpotentmodules} that $\g\CComodcon$ is equivalent to the category $\Ucon(\g)\CComod$ of comodules over the conilpotent universal enveloping coalgebra of $\g$. As such, it inherits the structure of a model category. Consider the algebra $\CE(\g)\otimes [\Ucon(\g)]^*$. Then it contains an MC element $\xi$, the image of the canonical MC element in $\g^*\otimes\CE(\g)$. We can then set $\CE(\g,M):=\CE(\g)\otimes M$, the Chevalley-Eilenberg complex of $\g$ with coefficients in $M$; note that this definition clearly agrees with that given in Section \ref{section:Koszul1}.

Similarly, given a $\CE(\g)$-module $N$, we consider $N^*\otimes [\Ucon(\g)]^*$. This is a
$\CE(\g)\hat{\otimes}[\Ucon(\g)]^*$-module and we can twist it by the element $\xi$.
The resulting pc $[\Ucon(\g)]^*$-module $(N^*\otimes [\Ucon(\g)]^*)^{[\xi]}$
corresponds to a $\g$-comodule with the underlying space $N\otimes \Ucon(\g)$ is called the conilpotent bar construction $\Bcon(N)$ of the $\CE(\g)$-module $N$.

The following result is an immediate corollary of definitions.
\begin{prop}\label{prop:coniladj}
	The functor $\Bcon:\CE(\g)\MMod\to \Ucon(\g)\CComod$ is right adjoint to
	$\CE(\g,-):\Ucon(\g)\CComod\to\CE(\g)\MMod$.
	The functor $\Bcon$ is (isomorphic to)  the functor associating to a $\CE(\g)$-module $N$ the maximal conilpotent submodule of $B(N)$.
\end{prop}
Informally, one can say that the conilpotent bar-cobar adjunction is obtained from the nonconilpotent one by restricting to conilpotent $\g$-comodules.

It is natural to ask whether the conilpotent Koszul duality constructed above can be promoted to a certain Quillen equivalence. For this we must (as before) impose some conditions on $\g$; indeed if $\g^*$ is a semisimple Lie algebra, then $\g$ has no conilpotent comodules and $\Ucon(\g)\cong \ground$. So we assume that $\g$ is itself conilpotent. Then we have the following result.

\begin{theorem}\label{thm:conilduality}
	The bar-cobar adjunction determines a Quillen equivalence between the category of conilpotent $\g$-comodules (equivalently $\Ucon(\g)$-comodules) and the category of $\CE(\g)$-modules where the latter has the standard projective model structure where weak equivalences are quasi-isomorphisms and fibrations are surjective maps.
\end{theorem}
\begin{proof} We may assume that $\ground$ is algebraically closed; the general case follows.
Consider the cobar construction  of the coalgebra $\Ucon(\g)$; recall that it has the form $\Omega(\Ucon(\g))\cong T\Sigma^{-1}\Ucon(\g)$ with the differential induced by that on $\Ucon(\g)$ and the coproduct on $\Ucon(\g)$. If $A$ is an algebra, then an algebra map $\Omega(\Ucon(\g))\to A$ is equivalent to an MC element in the convolution algebra $\Hom(\Ucon(\g), A)\cong A\hat{\otimes} \hat{\U}(\g^*)$. Thus, the canonical MC element in $\CE(\g)\hat{\otimes} \hat{\U}(\g^*)$ determines an algebra map $\Omega(\Ucon(\g))\to \CE(\g)$. It is now clear that the functor $M\mapsto \CE(\g,M)$ from conilpotent $\g$-comodules to $\CE(\g)$-modules factors as
\[
\Ucon(\g)\CComod\to  \Omega(\Ucon(\g))\MMod\to \CE(\g)\MMod
\]
where the first functor above is the part of a Quillen equivalence between comodules over a (conilpotent) coalgebra and modules over its cobar construction (see \cite[Section 8.4]{Posi}) and the second is the induction functor coming from the algebra map $\Omega(\Ucon(\g))\to \CE(\g)$ described above. Both functors are left Quillen, and so, in order to prove that their composition is a Quillen equivalence, it suffices to show that $\Omega(\Ucon(\g))\to \CE(\g)$ is a quasi-isomorphism. It is at this point that the condition of conilpotence of $\g$ is used. Using the filtration on $\g$ whose associated graded Lie coalgebra is abelian, the question is reduced to the case when $\g$ itself is abelian, and then, as in the proof of Theorem \ref{thm:coreflective}, further to the case when $\g$ is one-dimensional. In the one-dimensional situation there are two cases: when $\g$ is odd (and then $\Ucon(\g)$ is the exterior coalgebra on one generator) and when $\g$ is even  (and then $\Ucon(\g)$ is the polynomial coalgebra on one generator). (Note that in the conilpotent situation there is no special case when $\g$ is concentrated in degree zero.)  The desired statement follows from the Koszul resolution as in the non-conilpotent case.  Finally, if the abelian Lie coalgebra $\g$ is not finite-dimensional, it is a filtered colimit of finite-dimensional ones and the claim follows from passing to filtered colimits.
\end{proof}	
	\section{Koszul duality II: commutative cofibrant algebras}\label{section:cofibrant} Given a conilpotent  Lie coalgebra $\g$, its  Chevalley-Eilenberg complex $\CE(\g)$ is a cofibrant commutative augmented algebra (though it is not cofibrant as an associative dg algebra). This suggests constructing similarly a Koszul duality starting from the model category of the second kind of modules over  an augmented cofibrant commutative  algebra. This is what we will do in this section.
	\subsection{Bar-cobar adjunctions}\label{subsection:barcobar} Let $A$ be an augmented algebra with augmentation ideal $\bar{A}$. Let us recall that the \emph{bar construction} of $A$ is the  cofree conilpotent coalgebra $B(A):=(\Tcon(\Sigma \bar{A}), d_B)$ with the differential $d_B$ induced from the internal differential in $A$ and the multiplication in $A$. The \emph{extended} bar construction is the cofree (nonconilpotent) coalgebra $\check{B}(A):=(\Tc(\Sigma \bar{A}),d_{\check{B}})$ where the differential $d_{\check{B}}$ is similarly induced by the internal differential in $A$ and its associative multiplication; the details of this construction can be found in \cite{GL}. The extended bar construction has a  coaugmentation $\epsilon:\ground\to \check{B}(A)$, and $B(A)$ can be identified with the maximal conilpotent  subcoalgebra in $\check{B}(A)$ containing $\epsilon(\ground)$. Alternatively, $B(A)$ can be defined through its dual pc algebra $B^*(A)\cong \hat{T}(\Sigma^{-1}\bar{A}^*)$ and $\check{B}(A)$ through its dual pc algebra $\check{B}^*(A)\cong \check{T} (\Sigma^{-1}\bar{A}^*)$. The algebra $\check{B}^*(A)$ is augmented (where the augmentation ideal is generated by the elements in $\bar{A}^*$, and the local algebra $B^*(A)$  is obtained by localizing at the augmentation ideal.
	
	Now suppose that $A$ is commutative. In that case $B^*(A)$ is a complete Hopf algebra where the diagonal is specified by declaring the elements in $\bar{A}^*$ to be primitive: $\Delta(\alpha)=\alpha\otimes 1+1\otimes \alpha$ for $\alpha\in \Sigma^{-1}\bar{A}^*$; the compatibility of the differential $d_B$ with the diagonal follows from the commutativity of $A$. The bar construction $B(A)$  is then also a Hopf algebra with the cocommutative comultiplication being dual to the diagonal in $B^*(A)$.
	
	Next, the primitive elements $\operatorname{Prim}(B^*(A))$ form a (dg) Lie algebra and, since $B^*(A)$ is a completed tensor algebra on $\Sigma^{-1} \bar{A}^*$, this Lie algebra will be the completed free Lie algebra: $\operatorname{Prim}(B^*(A))\cong \hat{\Li}(\Sigma^{-1}  \bar{A}^*)$.
	\begin{defi}
		Let $A$ be an augmented commutative algebra.\begin{enumerate} \item
		The Harrison Lie coalgebra of $A$  is the Lie coalgebra dual to $\hat{\Li}(\Sigma^{-1} \bar{A}^*)$. It will be denoted by $\Harr(A)$.
	\item 	The universal enveloping coalgebra  of $\Harr(A)$ is called the modified bar construction of $A$:
	\[
	B'(A):=\Uc(\Harr(A)).
	\]
		\end{enumerate}
	\end{defi}
\begin{rem}
	It is clear that $\Harr(A)$, disregarding its differential, is the cofree conilpotent Lie coalgebra: $\Harr(A)\cong \Lcon(\Sigma \bar{A})$.  Considering the cofree (nonconilpotent) Lie coalgebra on $\Sigma \bar{A}$ leads to an extended version $\check{\Harr}(A)$ whose universal enveloping coalgebra is $\check{B}(A)$. We will not be concerned with this notion in the present paper and simply observe that there is a natural sequence of inclusions of  coalgebras:
	\[
	B(A)\subset B'(A)\subset \check{B}(A).
	\]
	Also note that $B'$ is a right adjoint functor from commutative algebras to Lie-conilpotent coalgebras whose left adjoint is the ordinary cobar construction $\Omega$; the proof is the same as for the corresponding statements about $B$ or $\check{B}$. cf. \cite[Section 6.1]{Posi} or \cite[Proposition 2.6]{GL}.
\end{rem}
Let $M$ be a module over a commutative algebra $A$; then ${B'}(A)\otimes M$ is a module over ${B'}^*(A)^{\op}\otimes {B'}^*(A)\otimes A$. Indeed, $B'(A)$ is a bimodule over ${B'}^*(A)$ which translates into an action of ${B'}^*(A)^{\op}\otimes {B'}^*(A)$ on $B'(A)$ and $A$ acts one the second tensor factor $M$.  Twisting by the canonical MC element $\xi\in {B'}^*(A)\otimes A$, we obtain a ${B'}^*(A)^{\op}\otimes [{B'}^*(A)\otimes A]^\xi$-module $({B'}(A)\otimes M)^{[\xi]}$. Disregarding the $[{B'}^*(A)\otimes A]^\xi$-module structure, we obtain a right ${B'}^*(A)$-module  structure on $({B'}(A)\otimes M)^{[\xi]}$ which comes from a  $B'(A)$-comodule (or, equivalently, a locally finite $\Harr(A)$-comodule).
\begin{defi}
	The locally finite $\Harr(A)$-comodule $({B'}(A)\otimes M)^{[\xi]}$ constructed above is called the modified bar construction of the $A$-module $M$. We will write $({B'}(A)\otimes M)^{[\xi]}$ as $B'(M)$.
\end{defi}
 Now let $N$ be a locally finite $\Harr(A)$-comodule (equivalently a $B'(A)$-comodule). Then $A\otimes N$  is a $A^{\op}\otimes A\otimes {B'}^*(A)^{\op}$-module (where $ A^{\op}\otimes A$ acts on $A$ and ${B'}^*(A)$ acts on the right on $N$) and we can twist it by the canonical MC element $\xi$ in  $ [A^{\op}\otimes {B'}^*(A)^{\op}]^{\xi}$ obtaining $(A\otimes N)^{[\xi]}$. Then $(A\otimes N)^{[\xi]}$ will be viewed as an $A$-module (forgetting the $ [A^{\op}\otimes {B'}^*(A)^{\op}]^{\xi}$-module structure).

 \begin{prop}
 	The functor $M\mapsto B'(M): A\MMod\to \Harr(A)\CComod$ is right adjoint to the functor $(A\otimes -)^{[\xi]}:\Harr(A)\CComod\to A\MMod$.
 \end{prop}
\begin{proof}
	Let $N$ be a $B'(A)$-comodule and $M$ be an $A$-module. Then, disregarding differentials, $\Hom_{B'(A)}(N, B'(M))\cong \Hom(N,M)$ since $B'(M)$ is a cofree $B'(A)$-comodule. Note that $\Hom(N,M)$ is a ${B'}^*(A)\otimes A$-module where ${B'}^*(A)$ acts on $N$ and $A$ acts on $M$; inspection shows that, as (dg) vector spaces we have $\Hom_{B'(A)}(N, B'(M))\cong [\Hom(N,M)]^{[\xi]}$. Similarly, $\Hom_A((N\otimes A)^{[{\xi}]},M)\cong  [\Hom(N,M)]^{[\xi]}$ and we are done.
\end{proof}	
We can now formulate our version of Koszul duality for the Harrison Lie coalgebra.
For its proof, recall that the definition of $\PerfII(A)$ extends to the case where $A$ itself is a dg category rather than a dg algebra, and moreover, given two dg algebras (or dg categories) $A$ and $A'$, we have
 a quasi-equivalence between $\PerfII(A\otimes A')$ and $\PerfII(\PerfII(A)\otimes \PerfII(A'))$ \cite[Lemma 4.16]{GHL}.

\begin{theorem}\label{thm:harrison}
	Let $A$ be a commutative augmented cofibrant algebra. \begin{enumerate}\item The pair of adjoint functors $(B',(-\otimes A)^{[\xi]})$ determines a Quillen adjunction between the model categories $A\MMod$ and $\Harr(A)\CComod$; moreover, the composition $B'\circ (-\otimes A)^{[\xi]}$ induces an endofunctor on the homotopy category of $\Harr(A)\CComod$ that is isomorphic to the identity functor.
	\item	If $A$ is connected (i.e. $A^0=\ground$ and $A^n=0$ for $n<0$), then the adjoint pair $(B',(-\otimes A)^{[\xi]})$ is a Quillen equivalence.
		\end{enumerate}
\end{theorem}
\begin{proof}
	It is immediate that the right adjoint functor $B'(-)$ takes fibrations of $A$-modules into fibrations of $B'(A)$-comodules whereas the left adjoint functor $(-\otimes A)^{[\xi]})$ takes cofibrations of $B'(A)$-comodules into cofibrations of $A$-modules.
	
	For (1) consider the cone $C(i)$ of the unit adjunction $i:N\to (B'(A)\otimes A\otimes N))^{[\xi\otimes 1+1\otimes \xi]}$. Recall that, disregarding the differential in $A$ and the bar-differential in $B'(A)$ there is an isomorphism $B'(A)\cong \Tcres(\Sigma\bar{A})$ and also $A\cong \bar{A}\oplus\ground$. On the other hand, there is a two-term resolution of the pc  $\Tres(\Sigma^{-1}\bar{A^*})$-module $N^*$, cf. Lemma \ref{lem:Lieconil} whose linear dual is isomorphic to $C(i)$. It follows that $C(i)$ is coacyclic (even absolutely acyclic) and so, $i$ is a weak equivalence of $B'(A)$-comodules. This proves (1).
	
	Now suppose that $A$ is connected, then $\Harr(A)$ is nonnegatively graded. We will use the fact that the pair of functors $(\Harr, \CE)$ form an adjunction between conilpotent Lie coalgebras and commutative algebras; for this see, e.g. \cite[Section 9]{LM} where, however a different terminology and notation was adopted in op.cit. The counit map of this adjunction is
	$\CE(\Harr(A))\to A$.
	
	Note the following commutative diagram
	\[
	\xymatrix{
\Omega B'(A)\ar^f[rr]\ar_{i_1}[dr]&&\CE(\Harr(A))\ar^{i_2}[dl]\\
&A	
}
	\]
where the downward arrows $i_1$ and $i_2$ are units of suitable adjunctions and the horizontal arrow $f$ is the canonical map $\Omega\Uc(\g)\to \CE(\g)$ defined for any Lie coalgebra $\g$ (in this case for $\g=\Harr(A)$).	We will consider the canonical element $\xi\in \MC(\Omega B'(A)\otimes B'(A))$ and its image $(f\otimes \operatorname{id})(\xi)\in \CE(\Harr(A))\otimes B'(A)$. It is clear that $(i_1\otimes \operatorname{id})(\xi)$ is the canonical MC element in $A\otimes B'(A)$, abusing the notation, we will write $\xi$ for either of these canonical MC elements.

 It follows that the functor $N\mapsto (N\otimes A)^{[\xi]}:\Harr(A)\MMod\to A\MMod$ can be factored as a composition \[N\mapsto \CE(\Harr(A), N)\cong (\CE(\Harr(A))\otimes N)^{[\xi]}\mapsto (A\otimes N)^{[\xi]}\] where the first functor is a Quillen equivalence $\Harr(A)\MMod\to\CE(\Harr(A))\MMod$ by Theorem \ref{thm:coreflective} and the second is induced by the counit map $\CE(\Harr(A))\to A$. Therefore, it suffices to show that the latter map induces an equivalence on $\PerfII$.

 Note that the map $\CE(\Harr(A))\to A$ is the standard cofibrant replacement of the cofibrant commutative algebra $A$. There are two model structures for  commutative algebras: one suitable for $\Z$-graded algebras \cite{Hinich} and the other admitting only non-negatively graded ones \cite{BG}; in either structure weak equivalences are simply quasi-isomorphisms. Since $A$ is connected, it is cofibrant in either model structure. Furthermore, it is known, cf. \cite[Proposition 7.11]{BG} that any connected cofibrant commutative algebra $X$ is isomorphic to the tensor product of a minimal algebra (determined up to an isomorphism by the weak equivalence class of $X$) and a collection of copies of contractible algebras of the form $S[x,y]$ with $dx=y$. We can thus write
 $A\cong M\otimes D$ and $ \CE(\Harr(A))\cong M\otimes D'$ where $D$ and $D'$ are contractible factors.

 Note that the algebra $S[x,y]$ with $dx=y$ and $|x|,|y|>0$ can be viewed as a pc algebra and thus, arguing as in the proof of Proposition \ref{prop:coderivedabelian}, we see that its $\PerfII$ is the same as that of $\ground$. It follows that
  $\PerfII(A)\simeq \PerfII(M\otimes D)\simeq \PerfII(M)$ and similarly $\PerfII(\CE(\Harr(A)))$ is also quasi-equivalent to $\PerfII(M)$. Moreover, clearly the canonical map $\CE(\Harr(A))\to A$ (which is a weak equivalence) induces an equivalence on $\PerfII$ as required.
\end{proof}	
\subsection{Conilpotent Koszul duality for commutative algebras} For competentness, we formulate the conilpotent Koszul duality for modules over commutative algebras; this is in some sense parallel to the material of Subsection 4.5 and is a straightforward consequence of standard associative Koszul duality. Let $A$ be an augmented commutative algebra, not necessarily cofibrant. We consider the category $A\MMod$ of $A$-modules.  Arguing as in
\S\ref{subsection:barcobar} we construct an adjoint pair of functors
$(B, (A\otimes-)^{[\xi]})$ between the category $A\MMod$ and the category of \emph{conilpotent} $\Harr(A)$-comodules. Note that the latter is isomorphic to the category of $\Ucon(\Harr(A))\CComod$ and as such, has a model structure.
\begin{theorem}\label{thm:conil2}
The adjoint pair
$(B, (A\otimes-)^{[\xi]})$ determines a Quillen equivalence between the category  $A\MMod$ with its standard projective model structure (so that weak equivalences are quasi-isomorphisms) and the category of \emph{conilpotent} $\Harr(A)$-comodules.
\end{theorem}
\begin{proof}
	The statement of the theorem follows at once from the statement of ordinary Koszul duality (Quillen equivalence) between $A\MMod$ and $B(A)\CComod=\Ucon(\Harr(A))\CComod$, cf. \cite[Theorem 6.3]{Posi} and \cite[Section 8.4]{Posi}.
\end{proof}

\subsection{Rational homotopy theory and $L_\infty$-local systems}\label{subsection:inftylocal} A large class of examples of cofibrant commutative algebras comes from rational homotopy theory \cite{BG}. Let $X$ be a topological space and $A(X)$ be its Sullivan minimal model, i.e. a cofibrant commutative algebra quasi-isomorphic to $C^*(X)$, the algebra of singular cochains on $X$. Note that $A(X)$ is an invariant of the weak homotopy type of $X$ and, if $X$ is nilpotent of finite type,  faithfully records its rational homotopy type. Thus, the homotopy category of $A(X)\MMod$ (or, equivalently, that of $\Harr(A(X))\CComod$) is a homotopy invariant of $X$. We will see that $A(X)$-modules are related to  $\infty$-local systems on $X$, also known as cohomologically locally constant sheaves on $X$, cf. \cite{BS, CHL, AVV} regarding this notion, at least when $X$ is a manifold. We will now briefly outline this relationship.

Suppose $X$ is a connected smooth manifold and $\operatorname{DR}(X)$ is the de Rham algebra of differential forms on $X$; we set $\ground=\mathbb{R}$. Then a finitely generated twisted $\operatorname{DR}(X)$-module or a retract of such determines an $\infty$-local system on $X$, cf. \cite{CHL}. Rational homotopy theory gives a quasi-isomorphism $A(X)\to \operatorname{DR}(X)$ and so, a functor $\PerfII(A(X))\to\PerfII(\operatorname{DR}(X))$. The objects of the dg category $\PerfII(A(X))$ will be called \emph{$L_\infty$-local systems} on $X$.  In other words, we obtain a functor $F$ from $L_\infty$-local systems on $X$ to $\infty$-local systems on $X$. This functor deserves further study which, however, would take us too far afield; so we will limit ourselves with a few simple observations.
\begin{itemize}
	\item The category $L_\infty(X):=\PerfII(A(X))$ is a homotopy invariant of $X$.
	\item Cofibrant (in the usual sense) $A(X)$-modules correspond to conilpotent $\Harr(A(X))$-modules; the corresponding $\infty$-local systems represent objects in $\PerfII(\operatorname{DR}(X))$. If such an $\infty$-local  system is given by a representation of $\pi_1(X)$, i.e. is an ordinary local system, this representation is given by unipotent matrices.
	\item Since the isomorphism class of $A(X)$ only records the nilpotent rational homotopy type of $X$, one does not expect a close relationship between $L_\infty$-local systems on $X$ and $\infty$-local systems on $X$ when $X$ is not nilpotent (e.g. when $\pi_1(X)$ is not nilpotent).
	\item Even when $X$ is nilpotent, the functor $F$ is not an equivalence on the level of homotopy categories, in general. Indeed, let $X$ be aspherical, i.e. all higher homotopy groups of $X$ vanish. In that case $\infty$-local systems on $X$ reduce to ordinary local systems on $X$, which in turn, are equivalent to flat vector bundles on $X$. Since the algebra $A(X)$ is local, every projective module over it is free, and so the resulting perfect $\operatorname{DR}(X)$-module will likewise be free. That means that the corresponding flat vector bundles on $X$ will be topologically trivial. In other words, local systems corresponding to topologically nontrivial flat bundles over $X$, are not in the image of $F$.
\end{itemize}
\begin{example}
Let $X=S^1$. In that case $A(X)$ is the graded algebra $H^*(X)\cong\Lambda(x)$ with $|x|=1$. Any MC element in $\Lambda(x)$ has the form $\xi_r:=rx$ where $r\in\mathbb{R}$. The corresponding twisted $\Lambda(x)$-module $\Lambda(x)^{[\xi]}$ gives rise to an $L_\infty$-local system on $S^1$ which is in this case an ordinary local system corresponding to the character $e^r$ of $\pi_1(S^1)\cong\mathbb{Z}$. The local system corresponding to the character $-1$ is not in the image of the functor $F$.

More generally, every $n$-dimensional $L_\infty$-local system on $S_1$ is an ordinary local system and it corresponds to an $n$-dimensional representation of $\mathbb{Z}$, i.e. an invertible $n\times n$ matrix $M$. Such a local system is in the image of $F$ if and only if $\det M>0$.
\end{example}

 \noindent
{\bf Acknowledgments. }The second author is grateful to Leonid Positselski for many useful discussions related to the subject of  present paper, and to the Jilin University where substantial part of this work was completed for hospitality and excellent working conditions.  

\end{document}